\documentclass[dvips]{imsart}

\RequirePackage[OT1]{fontenc}
\RequirePackage{amsthm,amsmath}
\RequirePackage[numbers]{natbib}
\RequirePackage[colorlinks,citecolor=blue,urlcolor=blue]{hyperref}

\arxiv{arXiv:0000.0000}

\startlocaldefs
\numberwithin{equation}{section}
\theoremstyle{plain}
\newtheorem{thm}{Theorem}[section]
\theoremstyle{remark}
\newtheorem{rem}{Remark}[section]
\theoremstyle{definition}
\newtheorem{corollary}{Corollary}[section]
\newtheorem{lemma}{Lemma}[section]

\endlocaldefs
\begin{document}

\begin{frontmatter}
\title{Cram\'{e}r's moderate deviations for  martingales with applications}
\runtitle{Cram\'{e}r's moderate deviations for  martingales}

\begin{aug}
\author{\fnms{Xiequan} \snm{Fan}\thanksref{a}\ead[label=e1]{fanxiequan@hotmail.com}}
\and
\author{\fnms{Qi-Man } \snm{Shao}\thanksref{b}%
\ead[label=e2]{shaoqm@sustech.edu.cn} }

\address[a]{Center for Applied Mathematics,
Tianjin University, Tianjin 300072,  China.\\
\printead{e1}}

\address[b]{Department of Statistics and Data Science, Southern University of Science and Technology,\\
 Shenzhen  518000, China.\\
\printead{e2} }

\runauthor{X. Fan and Q.M. Shao }

\affiliation{Tianjin University and Southern University of Science and Technology}

\end{aug}

\begin{abstract}
Let $(\xi_i,\mathcal{F}_i)_{i\geq1}$ be a  sequence of martingale differences.
Set $X_n=\sum_{i=1}^n \xi_i $ and $ \langle X \rangle_n=\sum_{i=1}^n \mathbf{E}(\xi_i^2|\mathcal{F}_{i-1}).$
We prove Cram\'er's moderate deviation expansions for $\displaystyle \mathbf{P}(X_n/\sqrt{\langle X\rangle_n} \geq x)$ and $\displaystyle \mathbf{P}(X_n/\sqrt{ \mathbf{E}X_n^2}  \geq x)$ as $n\to\infty.$
Our results extend  the classical Cram\'{e}r result to the cases of normalized martingales $X_n/\sqrt{\langle X\rangle_n}$ and standardized martingales $X_n/\sqrt{ \mathbf{E}X_n^2}$,
with martingale differences satisfying the conditional Bernstein condition.
Applications to  elephant random walks and  autoregressive processes are also discussed.
\end{abstract}


\begin{keyword}
\kwd{Martingales}
\kwd{Cram\'{e}r's moderate deviations}
\kwd{Berry-Esseen's bounds}
\kwd{Elephant random walks}
\end{keyword}

\end{frontmatter}

\section{Introduction}
Let $(\eta_i)_{i\geq 1}$ be a sequence of  independent and identically distributed (i.i.d.) centered real random variables. Denote $\sigma^2=\mathbf{E}\eta_{1}^2$ and $S_n=\sum_{i=1}^{n}\eta_{i}.$ The well-known central limit theorem (CLT) states that under the Lindeberg condition,
it holds
\[
\sup_{x \in  \mathbf{R}} \Big|\mathbf{P}(S_n/(\sigma\sqrt{n}) \leq x ) -\Phi(x) \Big | \rightarrow 0,\ \ \ \   \  n\rightarrow \infty,
\]
 where $\Phi(x)$  is the standard normal distribution function. The Berry-Esseen bound
 gives an estimation for the absolute error of the normal approximation: if $\mathbf{E}|\eta_1|^{3} < \infty,$ then
 \[
 \sup_{x \in  \mathbf{R}} \Big|\mathbf{P}(S_n/(\sigma\sqrt{n}) \leq x ) -\Phi(x) \Big |  \leq  \frac{c    }{\sqrt{n} } \frac{\mathbf{E}|\eta_1|^{3}}{  \sigma^3 },
 \]
 where $c$ is a positive absolute constant.
Cram\'{e}r's moderate deviation expansion stated below gives an estimation for the relative error of the normal approximation.
  Cram\'{e}r \cite{Cramer38} proved that if $\mathbf{E}\exp\{ c_{0}|\eta_{1}|\}<\infty $ for some constant $c_{0}>0,$
then it holds   for all $0\leq x =  o(\sqrt{n}\, ), $
\begin{equation}
\bigg| \ln \frac {\mathbf{P}(S_n/(\sigma\sqrt{n})> x)} {1-\Phi(x)} \bigg| =  O \bigg( \frac{1+x^3}{\sqrt{n}}\bigg),\ \ \ \ n\rightarrow \infty.
\label{cramer001}
\end{equation}
In particular, the last equality implies that for all $0\leq x =  o(n^{1/6}), $
\begin{equation}
\frac {\mathbf{P}(S_n/(\sigma\sqrt{n})> x)} {1-\Phi(x)}  = 1+o \big( 1\big), \ \ \ \ \  n \rightarrow \infty.
\label{cramer00g2}
\end{equation}
 Cram\'{e}r's moderate deviations type (\ref{cramer001})
for  independent random variables have been well studied. See,  for instance,  Petrov \cite{Pe54} and  Statulevi\v{c}ius \cite{S66}.
We refer to Chapter VIII of  Petrov \cite{Petrov75} and Saulis and Statulevi\v{c}ius \cite{SS78}  for further details on the subject.
Despite the fact that Cram\'{e}r's moderate deviations for independent random variables are well studied, there are only a few results of type (\ref{cramer001}) for martingales.
We refer to Bose \cite{Bose86a,Bose86b}, Ra\v{c}kauskas \cite{Rackauskas90,Rackauskas95,Rackauskas97}, Grama \cite{G97} and Grama and Haeusler \cite{GH00,GH06}.

Let $(\eta_i,\mathcal{F}_i)_{i\geq 1}$ be a sequence of  square integrable martingale differences defined on a probability space $(\Omega, \mathcal{F}, \mathbf{P})$.
Denote by $S_n=\sum_{i=1}^{n}\eta_{i},$ then $(S_i, \mathcal{F}_i)_{i\geq 1}$ is a martingale.
Assuming that there exist constants  $H, \sigma^2>0$ and $N\geq 0$ such that $\|\eta_{i}  \|_{\infty} \leq H$ and
\begin{equation}\label{ghgj}
 \bigg\| \sum_{i=1}^n \mathbf{E}(\eta _i^2|\mathcal{F}_{i-1})-n  \sigma^2 \bigg\|_{\infty} \leq N^2,
\end{equation}
Grama  and Haeusler \cite{GH00}   established the following Cram\'{e}r's moderate deviations for standardized martingales: for all $0 \leq x \leq \sqrt{\ln n} $,
\begin{equation}
\Bigg|\ln \frac {\mathbf{P}\left( S_n/(\sqrt{n}\sigma) >x\right)} {1-\Phi \left( x\right)} \Bigg|=
 O\bigg(  (1+x)\frac {\ln n}{\sqrt{n}} \bigg)
\label{cramer003}
\end{equation}
and, for  all $ \sqrt{\ln n} \leq x = o\left( n^{1/4}\right),$
\begin{equation}
\Bigg| \ln \frac {\mathbf{P}\left( S_n/(\sqrt{n}\sigma) >x \right)} {1-\Phi \left( x\right)} \Bigg| =
 O\bigg(
 \frac {x^3}{\sqrt{n}} \bigg),\ \ \  \   n\rightarrow \infty.
\label{cramer002}
\end{equation}
 By the exact convergence rates in martingale CLT of Bolthausen \cite{Bo82}, it is known that the term $\frac{\ln n}{\sqrt{n}}$ in (\ref{cramer003}) cannot be improved to $ \frac{1}{\sqrt{n}}.$ In Fan, Grama and Liu \cite{FGL13},   the expansions (\ref{cramer003}) and (\ref{cramer002}) have been extended  to the case of martingale differences $(\eta_i,\mathcal{F}_i)_{i\geq0}$ satisfying the conditional Bernstein condition
\begin{equation}
|\mathbf{E}(\eta_{i}^{k}  | \mathcal{F}_{i-1})| \leq \frac{1}{2}k!H^{k-2} \mathbf{E}(\eta_{i}^2 | \mathcal{F}_{i-1}),\ \ \ \ \ \mbox{for all}\ \ k\geq 3\ \ \mbox{and} \ \  i\geq 1,
\label{Bernst cond}
\end{equation}
and condition (\ref{ghgj}). Moreover, the range of validity  of (\ref{cramer002})  has been enlarged to $\sqrt{\ln n} \leq x = o\left( \sqrt{n} \right)$.
Notice that for i.i.d.\ random variables, Bernstein's condition (\ref{Bernst cond}) is equivalent to Cram\'{e}r's condition (cf.\ \cite{FGL13})
and therefore  (\ref{cramer002}) implies Cram\'{e}r's expansion (\ref{cramer001}).

Recently,  under  condition (\ref{ghgj}) and assumption that there exists a constant $\rho \in (0, 1]$ such that
\begin{equation}
\mathbf{E}(|\eta_{i}|^{2+\rho}  | \mathcal{F}_{i-1}) \leq H \  \mathbf{E}(\eta_{i}^2 | \mathcal{F}_{i-1}),\ \ \ \ \ \mbox{for all} \ \  i\geq 1,
\end{equation}
 Fan \emph{et al.}\ \cite{FGLS19} obtained Cram\'{e}r's moderate deviations for  self-normalized martingales $W_n:= \frac{S_n}{\sqrt{\sum_{i=1}^{n}\eta_{i}^2}}$.
In particular, they showed that
 for  all $0 \leq x = o( n^{\rho/(4\rho+2)} ),$ it holds
\begin{equation}
\frac {\mathbf{P}\left( W_n >x \right)} {1-\Phi \left( x\right)}= 1+o(1),\ \ \ \  \ \ n\rightarrow \infty.
\end{equation}
  See also  Fan \emph{et al.}\ \cite{FGLS20} for block type self-normalized martingales, where martingales satisfying condition (\ref{ghgj}) are also
  discussed.

 Though  Cram\'{e}r's moderate deviations for standardized  and self-normalized  martingales have been established,
the counterpart for normalized martingales  $U_n:=\frac{S_n}{ \sqrt{\sum_{i=1}^n \mathbf{E}(\eta _i^2|\mathcal{F}_{i-1})}}$
will put a new countenance for the study on the relative errors of the normal approximations.  In particular,  for i.i.d.\ random variables, we have $\sum_{i=1}^n \mathbf{E}(\eta _i^2|\mathcal{F}_{i-1})=n\sigma^2, $ and  normalized martingales become standardized sums $S_n/(\sigma\sqrt{n})$.
The main purpose of this paper is to establish Cram\'{e}r's moderate deviations for  normalized martingales
under   (\ref{Bernst cond}) and the following condition:
 for all $x>0,$
 \begin{equation}\label{tds}
 \mathbf{P}\bigg(\Big| \frac1n \sum_{i=1}^n \mathbf{E}(\eta _i^2|\mathcal{F}_{i-1})-   \sigma^2 \Big| \geq x \bigg) \leq  \frac{1}{H} \exp\bigg\{ - H x^2 n \bigg\}.
\end{equation}
Notice  that with condition (\ref{ghgj}), it is possible to obtain Cram\'{e}r's moderate deviations for normalized martingales via
the result of Fan, Grama and Liu \cite{FGL13}.
However, condition (\ref{ghgj})  is quite restrictive, and is not satisfied for some natural models, for instance, elephant random walks and autoregressive processes. It  is more nature to assume condition  (\ref{tds}) for the settings mentioned above.
 Hence, we prefer to establish Cram\'{e}r's moderate deviations under condition (\ref{tds}) rather than (\ref{ghgj}).
 From Theorem \ref{th0} below, it follows that  for all $0 \leq  x   =o(  \sqrt{n}  ),$
\begin{equation}\label{dfsdsd}
 \Bigg| \ln \frac{\mathbf{P}(U_n  >x)}{1-\Phi \left( x\right)} \Bigg|\leq c\, \Bigg(    \frac{x^3 }{\sqrt{n} }    + (1+x) \frac{ \ln n }{\sqrt{n} }   \Bigg).
\end{equation}
In particular, it implies that for all $0 \leq  x   =o(  n^{1/6}),$
\begin{equation}\label{fsddsf}
 \frac{\mathbf{P}(U_n  >x)}{1-\Phi \left( x\right)} =1+o(1),\ \  \ \ \ \ \  n\rightarrow \infty.
\end{equation}
Notice that  the range of validity of our expansion (\ref{fsddsf})
 is same to the one of Cram\'{e}r's result (\ref{cramer00g2}).
 Moreover,  inspecting the proof of Theorem 2.1, we conclude that (\ref{dfsdsd}) holds true
 when $U_n$ is replaced by the standardized martingale $S_n/(\sqrt{n}\sigma)$, under the conditions
 (\ref{Bernst cond}) and (\ref{tds}).

The paper is organized as follows. Our main results are stated and discussed in Section \ref{sec2}.
Applications of our results to elephant random walks and  autoregressive processes
are discussed in Section \ref{sec3}. The remaining sections are devoted to the proofs of theorems.

Throughout the paper, $c$ and $c_\alpha,$ probably supplied with some indices,
denote respectively a generic positive absolute constant and a generic positive constant depending only on $\alpha.$
Their values may vary from line to line.

\section{Main results}  \label{sec2}
\setcounter{equation}{0}

Let $(\xi _i,\mathcal{F}_i)_{i=0,...,n}$ be a finite sequence of martingale differences, defined on a
 probability space $(\Omega ,\mathcal{F},\mathbf{P})$,  where $\xi
_0=0 $,  $\{\emptyset, \Omega\}=\mathcal{F}_0\subseteq ...\subseteq \mathcal{F}_n\subseteq
\mathcal{F}$ are increasing $\sigma$-fields and $(\xi _i)_{i=1,...,n}$ are allowed to depend on $n$. Set
\begin{equation}
X_{0}=0,\ \ \ \ \ X_k=\sum_{i=1}^k\xi _i,\quad k=1,...,n.  \label{xk}
\end{equation}
Then $(X_i,\mathcal{F}_i)_{i=0,...,n}$ is a martingale. Denote by $\left\langle X\right\rangle $   the quadratic characteristic of the
martingale $X=(X_k,\mathcal{F}_k)_{k=0,...,n},$ that is
\begin{equation}\label{quad}
\left\langle X\right\rangle _0=0,\ \ \ \ \ \ \ \ \left\langle X\right\rangle _k=\sum_{i=1}^k\mathbf{E}(\xi _i^2|\mathcal{F}
_{i-1}),\ \ \ \ \quad k=1,...,n.
\end{equation}
In the sequel we shall use the following conditions:

\begin{description}
\item[(A1)]  There exists  a number $\epsilon_n \in (0, \frac12] $   such that
\[
|\mathbf{E}(\xi_{i}^{k}  | \mathcal{F}_{i-1})| \leq \frac12 k!\epsilon_n^{k-2} \mathbf{E}(\xi_{i}^2 | \mathcal{F}_{i-1}),\ \ \ \ \ \textrm{for all}\ k\geq 2\ \ \textrm{and}\ \ 1\leq i\leq n;
\]
\item[(A2)]  There exist  a number  $ \delta_n\in (0, \frac12]$  and a positive constant $C$ such that for all $  x >0 ,$
\[
 \mathbf{P}( \left| \left\langle X\right\rangle _n-1\right|  \geq x   )  \leq  C \exp\{  - x^2  \delta_n^{-2} \}.
 \]
\end{description}

It is remarked that as $\mathbf{E} \left\langle X\right\rangle _n =\sum_{i=1}^n \mathbf{E} \xi_{i}^{2} $,   condition (A2) implies that $(\xi _i,\mathcal{F}_i)_{i=0,...,n}$ is standardized, that is, $\sum_{i=1}^n \mathbf{E} \xi_{i}^{2} $
is close to $1$.
Clearly, in the case of standardized sums of i.i.d.\ random variables, with  variances  $\sigma^2 >0$ and finite moment generating functions, the conditions (A1) and (A2) are satisfied with $\epsilon_n = O(\frac {1} { \sqrt n})$ and $\delta_n = O(\frac1{\sqrt{n}})$. In the case of martingales, we assume that  $\epsilon_n$ and $\delta_n$  depend on $n$  such that $\epsilon_n ,  \delta_n \rightarrow 0, \   n\rightarrow \infty.$    

The following theorem gives a  Cram\'{e}r's moderate deviation expansion  for normalized martingales.
\begin{thm}\label{th0}
Assume that conditions (A1) and (A2) are satisfied.  Then for all $0 \leq  x   =o( \min\{\epsilon_n^{-1}, $ $  \delta_n ^{-1 } \}),$
\begin{equation}\label{t0ie1}
 \bigg| \ln \frac{\mathbf{P}(X_n/\sqrt{ \langle X\rangle_n } >x)}{1-\Phi \left( x\right)} \bigg|\leq c  \bigg( x^3  (\epsilon_n  + \delta_n)+   (1+  x )  \big(  \delta_n|\ln \delta_n| + \epsilon_n|\ln \epsilon_n| \big) \bigg).
\end{equation}
Moreover, the same inequality holds true when $X_n /\sqrt{ \langle X\rangle_n }$ is replaced by $-X_n /\sqrt{ \langle X\rangle_n }$.
\end{thm}

Notice that $e^{ x }=1+O(|x|)$ for all $|x|=O(1).$
From Theorem \ref{th0}, we have the following  result  about the equivalence to the normal tail.
\begin{corollary}\label{corollary01}
Assume that conditions (A1) and (A2) are satisfied.  Then for all $0 \leq  x   =O( \min\{\epsilon_n^{-1/3}, $ $  \delta_n ^{-1/3 } \}),$
\begin{equation}\label{gs1ddsg}
 \frac{\mathbf{P}(X_n /\sqrt{ \langle X\rangle_n } >x)}{1-\Phi \left( x\right)} = 1+   O  \bigg( x^3  (\epsilon_n  + \delta_n)+   (1+  x )  \big(  \delta_n|\ln \delta_n| + \epsilon_n|\ln \epsilon_n| \big) \bigg).
\end{equation}
Then for all $0\leq x  =o( \min\{  \epsilon_n^{-1/3 }, \delta_n^{-1/3 }\}),$
\begin{equation}\label{gsdgssd}
\frac{\mathbf{P}(X_n /\sqrt{ \langle X\rangle_n } >x)}{1-\Phi \left( x\right)}=1+o(1) .
\end{equation}
Moreover, the same  equalities hold  true when $X_n /\sqrt{ \langle X\rangle_n }$ is replaced by $-X_n /\sqrt{ \langle X\rangle_n }$.
\end{corollary}

From Theorem \ref{th0},
by an argument similar to the proof of Corollary 2.2 in Fan \emph{et al.}\,\cite{FGLS19}, we easily obtain the following moderate deviation principle  (MDP) result.   Such type results for  \emph{standardized} martingales
 have been established by Gao \cite{G96}, Worms \cite{W01} and Djellout \cite{D02}.   See also  Dedecker \emph{et al.}\,\cite{DMPU09} for stationary sequences.

\begin{corollary}\label{corollary02}
Assume that   conditions (A1) and (A2) are satisfied.
Let $a_n$ be any sequence of real numbers satisfying $a_n \rightarrow \infty$ and $a_n   \min\{\epsilon_n,  \delta_n    \} \rightarrow 0$
as $n\rightarrow \infty$.  Then  for each Borel set $B$,
\begin{eqnarray*}
- \inf_{x \in B^o}\frac{x^2}{2} &\leq & \liminf_{n\rightarrow \infty}\frac{1}{a_n^2}\ln \mathbf{P}\left(\frac{1}{a_n} \frac{X_n}{ \sqrt{ \langle X\rangle_n } }  \in B \right) \\
 &\leq& \limsup_{n\rightarrow \infty}\frac{1}{a_n^2}\ln \mathbf{P}\left(\frac{1}{a_n} \frac{X_n}{ \sqrt{ \langle X\rangle_n } } \in B \right) \ \leq \  - \inf_{x \in \overline{B}}\frac{x^2}{2} \, ,
\end{eqnarray*}
where $B^o$ and $\overline{B}$ denote the interior and the closure of $B$, respectively.
\end{corollary}

The exact Berry-Esseen bounds for \emph{standardized} martingales under various moment conditions has been established. We refer to   Bolthausen \cite{Bo82}, Haeusler \cite{H88},  El Machkouri and Ouchti \cite{EO07}, Mourrat \cite{M13}, \cite{F19} and  Dedecker \emph{et al.}\,\cite{DMR22}. The following  theorem gives a Berry-Esseen's bound  for normalized  martingales
\begin{corollary}\label{corollary03}
Assume that   conditions (A1) and (A2) are satisfied.  Then it holds
\begin{equation}
\sup_{x \in \mathbf{R}}\Big| \mathbf{P}(X_n /\sqrt{ \langle X\rangle_n } \leq x)- \Phi(x)\Big| \ \leq \ c \, \Big(  \delta_n|\ln \delta_n| + \epsilon_n|\ln \epsilon_n| \Big).
\end{equation}
\end{corollary}

By inspecting the proof of Theorem \ref{th0}, it is easy to see that Theorem \ref{th0} holds true when
$\mathbf{P}(X_n /\sqrt{ \langle X\rangle_n } >x)$   is replaced by $\mathbf{P}(X_n >x ).$
\begin{thm}\label{th2s}
Assume that conditions (A1) and (A2) are satisfied.  Then for all $0 \leq  x   =o( \min\{\epsilon_n^{-1}, $ $ \delta_n ^{-1 } \}),$
\begin{equation}\label{t0ie1}
 \Bigg| \ln \frac{\mathbf{P}(X_n >x )}{1-\Phi \left( x\right)} \Bigg|\leq c_{ p}   \bigg( x^3  (\epsilon_n  + \delta_n)+   (1+  x )  \big(  \delta_n|\ln \delta_n| + \epsilon_n|\ln \epsilon_n| \big) \bigg) .
\end{equation}%
\end{thm}

With   conditions  (A1) and   $\| \left\langle X\right\rangle _n-1\|_\infty \leq \delta_n^2 $,   results similar to (\ref{t0ie1}) can be found
 in Fan \emph{et al.}\,\cite{FGL13}. Thus  the last theorem can be regarded as an extension of the main result of Fan \emph{et al.}\,\cite{FGL13}.

\begin{rem}\label{fdsdds}
From the proof of Theorem \ref{th0}, it is easy to see that Theorems  \ref{th0} and  \ref{th2s}  hold true when condition (A2) is
replaced by  the following condition:   There exist  a number  $ \delta_n\in (0, \frac12]$  and a constant $c$ such that for all $  x > 0 ,$
\[
 \mathbf{P}( \left| \left\langle X\right\rangle _n-1\right|  \geq x   )  \leq  c \, \exp\{  - x   \delta_n^{-2} \}.
 \]
Clearly, when $0<x \leq 1,$ the last condition implies condition (A2).
\end{rem}

\begin{rem} Instead of standardized  martingales, we consider the general martingales.
Assume that $\xi_{i}=\eta_{i}/(\sqrt{n}\sigma_n)$, where $(\eta _i,\mathcal{F}_i)_{i \geq 1}$ is a sequence of martingale differences satisfying the following three conditions:
\begin{description}
\item[(A1$'$)] There exist two constants $A, B > 0$ such that for all $n\geq 1,$
\[
A \leq \sigma_n^2:=  \frac1{n }\sum_{i=1}^{n}\mathbf{E} \eta^2_{i}  \leq B;
\]
\item[(A2$'$)] (\emph{Conditional Bernstein's condition}) There exists a  constant $H >0$ such that for all $i\geq 1,$
\[
\Big|\mathbf{E}(\eta_{i}^{k}  | \mathcal{F}_{i-1}) \Big| \leq \frac12 k!H^{k-2} \mathbf{E}(\eta_{i}^2 | \mathcal{F}_{i-1}), \    \ \  \  k\geq 3;
\]
\item[(A3$'$)] There exist  two  constants $  C_1, C_2 > 0$ such that   for all $n$ and $  x > 0 ,$
\[
\mathbf{P}\bigg( \Big|  \frac1{n \sigma_n^2}\sum_{i=1}^{n}\mathbf{E}(\eta^2_{i} | \mathcal{F}_{i-1}) -1\Big|  \geq x \bigg)  \leq  C_1 \exp\bigg\{  - C_2\,x^2\,n\,\bigg \}.
 \]
\end{description}
It is known that if  $(\eta_{i})_{i\geq 1} $ are i.i.d.\ random variables with  finite  exponential moments, the conditions  (A1$'$), (A2$'$) and (A3$'$) are satisfied; see Proposition 8.2 in \cite{FGL13}.   Denote by
\[
U_n=\frac{\sum_{i=1}^{n}\eta_{i}}{ \sqrt{\sum_{i=1}^{n}\mathbf{E}(\eta^2_{i} | \mathcal{F}_{i-1})}}
\]
the normalized martingale. Then, by Theorem \ref{th0}, we have  for all $0 \leq  x   =o(  \sqrt{n}  )$ as  $ n \rightarrow \infty$,
\begin{equation}\nonumber
 \Bigg| \ln \frac{\mathbf{P}(U_n  >x)}{1-\Phi \left( x\right)} \Bigg|\leq c  \Bigg(    \frac{x^3 }{\sqrt{n} }    + (1+x) \frac{ \ln n }{\sqrt{n} }   \Bigg).
\end{equation}
In particular, it implies that for   all $0\leq x  =o(n^{1/6})$ as $n\rightarrow \infty,$
\begin{equation}\nonumber
 \frac{\mathbf{P}( U_n>x  )}{1-\Phi \left( x\right)}=1+o(1).
\end{equation}
Moreover, by Theorem \ref{th2s}, the same results hold true when $U_n$ is replaced by $ \sum_{i=1}^{n}\eta_{i} / (\sqrt{n} \sigma_n) $.
\end{rem}

%

\section{Applications} \label{sec3}
\setcounter{equation}{0}

\subsection{Elephant random walks}
 The elephant random walk (ERW) is a type of one-dimensional random walk on integers, which has a complete memory of its whole history.  It was first introduced in 2004 by Sch{\"u}tz and Trimper \cite{schutz2004elephants}  in order to study the memory effects in the non-Markovian random walk, and has then raised much
interest.  The model can be described as follows.   The ERW starts  at time $n=0$, with  position $T_0=0$. At time $n=1$, the elephant
 moves to $1$ with probability $1/2$ and to $-1$ with probability $ 1/2$.
 So the position of the elephant at time $n=1$ is given by $T_1=X_1,$ with $X_1$
 a Rademacher $\mathcal{R}(q)$ random variable.
 At time $n,$ for $n\geq 2,$ an integer $n'$ is chosen from the set $\{1, 2,\ldots, n-1\}$
 uniformly at random. Then $X_{n}$ is determined stochastically by the following rule:
 \begin{displaymath}
 X_{n } =   \left\{ \begin{array}{ll}
X_{n'} & \textrm{with  probability  $p$}\\
-X_{n'} & \textrm{with  probability  $1-p$}.
\end{array} \right.
\end{displaymath}
In other words,  at time $n,$ we reinforce $X_{n'}$ with  probability $p$ and reduce $X_{n'}$ with  probability $1-p$.
Thus, for $n\geq 2,$  the position of the elephant at time $n $ is
 \begin{eqnarray}\label{adfs0}
\   T_{n }=\sum_{i=1}^{n }X_i,
\end{eqnarray}
 with
 \begin{eqnarray}
\    X_{n }=\alpha_nX_{\beta_n}, \nonumber
\end{eqnarray}
where $\alpha_n$ has a Rademacher distribution $\mathcal{R}(p)$, $p\in [0, 1],$ and $\beta_n$ is
random with the uniform distribution on the integers $\{1, 2,\ldots, n-1\}$. Moreover, $\alpha_n$ is independent of $X_{1},...,X_{n}$, and $\beta_1, \beta_2, . . .$
are independent.
Here  $p $ is called  the memory parameter. The ERW is respectively  called diffusive,  critical and   superdiffusive according to the memory parameter $p \in [0, 3/4),$ $p=3/4$ and $p \in (3/4, 1]$.

 The description of the asymptotic behavior of the ERW has
motivated many interesting works. Baur and Bertoin \cite{baur2016elephant} established the functional limit theorem via a method of connection to P\'{o}lya-type urns.
Coletti, Gava and Sch{\"u}tz \cite{coletti2017central,CGS17} derived the  CLT  and a strong invariance principle for $p \in [0,  3/4]$ and a law of large numbers for $p \in [0, 1).$ Moreover, they also showed that if $p \in (3/4, 1]$, then the ERW converges to a non-degenerate random variable which is not normal. V\'{a}zquez Guevara \cite{G19} obtained the almost sure CLT.  Bercu \cite{B18e} recovered the CLT via a martingale method.
 Bercu and Lucile \cite{BL19} introduced a multi-dimensional ERW,  and gave a multivariate CLT.
  Fan \emph{et al.}\ \cite{FHM20} obtained some Cram\'{e}r's moderate deviations.
 Recently, Bertoin \cite{B21} studied the memory impacts passages at the origin for ERW in the diffusive regime.

In this subsection, we introduce a generalization of ERW such  that the step sizes varying in time. Let $(Z_i)_{i\geq1}$ be a sequence of positive and i.i.d.\ random variables, with finite means $\nu=  \mathbf{E} [Z_1]$ and variances $\textrm{Var}(Z_1)=\sigma^2>0$. Moreover, $(Z_i)_{i\geq1}$ is independent of $(X_i)_{i\geq1}$. An ERW with  random step sizes can be described as follows. At time $n = 1$, the elephant
 moves to $Z_1$ with probability $1/2$ and to $-Z_1$ with probability $1/2$.
 So the position $Y_1$ of the elephant at time $n=1$ is given by   the following rule:
 \begin{displaymath}
 Y_{1} =   \left\{ \begin{array}{ll}
Z_1 & \textrm{\ \ \ with  probability  $1/2$}\\
-Z_1 & \textrm{\ \ \ with  probability  $1/2$}.
\end{array} \right.
\end{displaymath}
For $n\geq2,$  instead of (\ref{adfs0}), the position of the elephant with random step sizes at time $n$ is
\[
S_{n}=\sum_{i=1}^{n}Y_i,
\]
 where
\begin{eqnarray}
 Y_{n}=\alpha_{n }X_{\beta_{n }}Z_{n}.
\end{eqnarray}
Notice that $|Y_{n}|=Z_{n}$ for all $n\geq 1$. Thus at time $n$, the step size is $Z_{n},$ which is a random variable.
Without loss of generality, we may assume that $\nu=1$. Otherwise, we may consider the case $S_{n}/\nu$ instead of $S_{n}.$
Clearly, when $\sigma=0,$ the ERW with  random step sizes reduces to the usual ERW.

The next theorem gives some Cram\'{e}r's  moderate deviations for the ERW with uniformly bounded random step sizes.
For $p \in (0, 3/4],$  denote for all $n\geq1,$
\[
\ \ a_n=\frac{\Gamma(n)\Gamma(2p)}{\Gamma(n+2p-1)} \ \ \ \  \    \textrm{and} \ \   \ \ \    v_n=\sum_{i=1}^n a_i^2.
\]
Notice that the exact values of $a_n$ and $v_n$ can be easily calculated
via computer.
\begin{thm}\label{thmg2} Assume that $0\leq Z_1\leq C$ and $p \in (0,  3/4]$.
 The following inequalities hold.
 \begin{description}
  \item[\textbf{[i]}]  If $p \in (0,  1/2]$,  then  for all $0\leq x =o( \sqrt{n}),$
	\begin{equation}\label{fth1}
\Bigg| \ln \frac{\mathbf{P}\Big(a_n S_n \geq x\sqrt{v_n  + na_n^2 \sigma^2 }\, \Big)}{1-\Phi \left( x\right)} \Bigg| \leq   c  \bigg(\frac{x^3}{\sqrt{n}}   +  (1+x) \frac{ \ln n }{\sqrt{n}}  \bigg) .
\end{equation}
  \item[\textbf{[ii]}]  If $p \in (1/2, 3/4)$,  then   for all $0\leq x   =o(   n^{ (3-4p) /2 }),$
	\begin{equation}
\Bigg| \ln\frac{\mathbf{P}\Big(a_n S_n \geq x\sqrt{v_n  + na_n^2 \sigma^2 }\, \Big)}{1-\Phi \left( x\right)} \Bigg| \leq     c_{p} \bigg(\frac{x^3}{    n^{ (3-4p) /2 } }  +  (1+x) \frac{ \ln n }{   n^{ (3-4p) /2 } } \bigg)  .
\end{equation}

  \item[\textbf{[iii]}] If $ p=3/4$,  then   for all $0\leq x =o(\sqrt{\ln n}),$
	\begin{equation}
\Bigg| \ln \frac{\mathbf{P}\Big(a_n S_n \geq x\sqrt{v_n  + na_n^2 \sigma^2 }\, \Big)}{1-\Phi \left( x\right)}  \Bigg| \leq   c \bigg(\frac{x^3}{\sqrt{ \ln n}}  +  (1+x) \frac{ \ln \ln n  }{\sqrt{\ln n}}  \bigg)  .
\end{equation}
\end{description}
Moreover, the same inequalities hold when $ a_n S_n  $ is replaced by $ -a_n S_n  $. In particular, these inequalities imply that
for $p \in (0,  3/4),$
\begin{equation}
\frac{\mathbf{P}\Big( a_n S_n \geq x \sqrt{ v_n  + na_n^2 \sigma^2 } \, \Big)}{1-\Phi \left( x\right)}=1+o(1) \ \ \ and \ \ \ \frac{\mathbf{P}\Big( a_n S_n \leq- x \sqrt{v_n  + na_n^2 \sigma^2 } \, \Big)}{\Phi \left( -x\right)}=1+o(1)
\end{equation}
uniformly for $0\leq x  =o( \min\{ n^{1/6},  n^{(3-4p)/6}\}).$
\end{thm}

From Theorem \ref{thmg2},  following an argument similar to the proof of Corollary \ref{corollary03}, we have the following Berry-Esseen bounds.
\begin{corollary} \label{cor01}
 Assume that $0 \leq Z_1\leq C $ and $p \in (0,  3/4]$.
The following inequalities hold.
\begin{description}
  \item[\textbf{[i]}]  If $  p \in (0, 1/2],$  then
	\begin{equation}\label{ineq2}
	\sup _{x\in\mathbf{R}} \bigg|\mathbf{P}\Big( \frac{a_n S_n}{\sqrt{v_n  + na_n^2 \sigma^2 }}    \leq x \Big)-\Phi(x) \bigg|\leq    C  \frac{\ln n }{ \ \sqrt{n} \   }  .
	\end{equation}
 \item[\textbf{[ii]}]  If $   p \in (1/2, 3/4),$  then
	\begin{equation}\label{ineq3}
	\sup _{x\in\mathbf{R}} \bigg|\mathbf{P}\Big( \frac{a_n S_n}{\sqrt{v_n  + na_n^2 \sigma^2 }}    \leq x \Big)-\Phi(x) \bigg| \leq    C_p  \frac{\ln n }{    n^{ (3-4p) /2 } \ } .
	\end{equation}
  \item[\textbf{[iii]}] If $p=3/4,$ then
   \begin{equation}\label{ineq4}
   \sup _{x\in\mathbf{R}} \bigg|\mathbf{P}\Big( \frac{a_n S_n}{\sqrt{v_n  + na_n^2 \sigma^2 }}    \leq x \Big)-\Phi(x) \bigg| \leq C \frac{ \ \ln \ln n  \    }{\sqrt{\ln n}  }.
   \end{equation}	
\end{description}
\end{corollary}

From the last corollary, we obtain the following CLT for the ERW with uniformly bounded random step sizes: If  $  p \in (0,  3/4]
$, then
	\begin{equation}\label{ap01}
	\frac{a_n S_n}{\sqrt{v_n  + na_n^2 \sigma^2 }} \stackrel{\mathbf{D}}{\longrightarrow} \mathcal{N}(0, 1),\ \ \ n\rightarrow \infty,
	\end{equation}
where $\stackrel{\mathbf{D}}{\longrightarrow}$ stands for convergence in distribution.
For $p \in (0,    3/4),$  we have
\[
  \displaystyle \lim_{n\rightarrow \infty} \frac{\sqrt{v_n  + na_n^2 \sigma^2 } } { a_n\sqrt{n(\sigma^2+ 1/(3-4p))}}  =1
\]
(cf.  the inequalities  (\ref{a10}) and (\ref{a15})). Thus, (\ref{ap01}) also implies that  for $  p \in (0,    3/4),$
\[
	\frac{ S_n}{\sqrt{n(\sigma^2+1/(3-4p)) }} \stackrel{\mathbf{D}}{\longrightarrow} \mathcal{N}(0, 1),\ \ \ n\rightarrow \infty.
\]

Assume that  $\sigma^2$ is unknown, but
the sept sizes $(Z_{n})_{n\geq 1}$ are observable. Then we have the following self-normalized type  Cram\'{e}r's  moderate deviations.
\begin{thm}\label{thmg345}
 Theorem \ref{thmg2} and Corollary  \ref{cor01}
remain valid  when $v_n  + na_n^2 \sigma^2$ is replaced by $v_n  +a_n^2 \sum_{i=1}^n (Z_i-1)^2$.
In particular, for any $p \in (0,  3/4)$, the following equalities hold
\begin{equation} \nonumber
\frac{\mathbf{P}\Big(  a_n S_n   \geq x  \sqrt{v_n  +a_n^2 \sum_{i=1}^n (Z_i-1)^2 } \, \Big)}{1-\Phi \left( x\right)}=1+o(1)
\end{equation}
and
\begin{equation} \nonumber
\frac{\mathbf{P}\Big( a_n S_n \leq- x \sqrt{v_n  +a_n^2 \sum_{i=1}^n (Z_i-1)^2 } \, \Big)}{\Phi \left( -x\right)}=1+o(1)
\end{equation}
uniformly for $0\leq x  =o( \min\{ n^{1/6},  n^{(3-4p)/6}\}).$
\end{thm}

Self-normalized type  Cram\'{e}r's  moderate deviations is  user-friendly since
in practice one usually does not know the exact value of  $\sigma^2.$

\begin{rem}
If at time $n = 1$, the elephant
 moves to $Z_1$ with probability $q$ and to $-Z_1$ with probability $1-q$ for some $q \in [0, 1]$, then Theorems \ref{thmg2} and \ref{thmg345}  remain  valid.
 The proofs are similar, but $M_n$ should be redefined as  $a_n S_n-2q+1$ and notice that for all $0 \leq x =o(   \varepsilon_n^{-1 }  ) , \varepsilon_n \searrow 0,$ it holds
\begin{equation}\label{f42}
1-\Phi \left(x + \varepsilon_n \right) =\Big( 1-\Phi (x)\Big)\exp \Big\{  \theta  c \, (1+x)\varepsilon_n \Big\},
\end{equation}
where $|\theta|\leq 1.$
\end{rem}

\subsection{Autoregressive processes}
The autoregressive processes can be described as follows: for all $n\geq 0$,
\begin{equation}\label{eqAuto}
X_{n+1}=\theta X_n+\varepsilon_{n+1},
\end{equation}
where $\theta$, $X_n$ and $\varepsilon_n$ are   respectively an unknown parameter,  the observations and driven noises. We assume that $(\varepsilon_n)_{n\geq 0}$ is a sequence of i.i.d.\ centered random variables with finite variations  $\mathbf{E} \varepsilon_{0}^2= \sigma^2 > 0$ and that $X_0=\varepsilon_0$.  The  unknown parameter $\theta$   can be estimated    by the following least-squares estimator  for all $n\geq 1$,
\begin{equation}\label{eqEstimate}
\hat{\theta}_n=\frac{ \sum_{k=1} ^n X_{k-1} X_k}{ \sum_{k=1} ^n X_{k-1}^2}.
\end{equation}
 When $(\varepsilon_n)_{n\geq0}$ are normal random variables,  large deviation principles for the case $|\theta| < 1$
were established in Bercu \emph{et al.}\ \cite{BGR97}, and exponential inequalities for the deviation of $\hat{\theta}_n-\theta$  have been established  in  Bercu and Touati\ \cite{BDR15}. See also Jiang \emph{et al.}\ \cite{JWY22} for the explosive autoregressive processes.
In the following theorem, we give a self-normalized Cram\'{e}r's moderate deviation result for $\hat{\theta}_n-\theta$, provided that the driven noises are bounded.

\begin{thm}\label{ththeta2}
Assume   $|\varepsilon_0 |  \leq H$ for some positive constant $H$.   If $|\theta| < 1 $, then
 for all $0 \leq  x   =o(  \sqrt{n}  ),$
\begin{equation}\label{gsdsd1}
 \Bigg| \ln \frac{\mathbf{P} \big(   (\hat{\theta}_n -\theta  )\sqrt{ \Sigma_{k=1}^n X_{k-1}^2} > x \sigma \big)}{1- \Phi(x)}  \Bigg|\leq c \, \Bigg(    \frac{x^3 }{\sqrt{n} }    + (1+x) \frac{ \ln n }{\sqrt{n} }   \Bigg).
\end{equation}
In particular, it implies that for all  $ 0 \leq x =o(n^{1/6})$,
\begin{eqnarray}
  \frac{\mathbf{P} \big(   (\hat{\theta}_n -\theta  )\sqrt{ \Sigma_{k=1}^n X_{k-1}^2} > x \sigma \big)}{1- \Phi(x)} =1+o(1),\nonumber\ \ \  n\rightarrow \infty.
\end{eqnarray}
Moreover, the results remain valid when $ \hat{\theta}_n -\theta $ is replaced by $\theta-\hat{\theta}_n$.
\end{thm}


 By Theorem \ref{ththeta2}, it is easy to establish
the following confidence intervals  for the parameter $\theta$.
\begin{corollary}\label{c0kls} Assume the condition of Theorem \ref{ththeta2}.
Let $\kappa_n \in (0, 1).$  Assume that
\begin{eqnarray}\label{gdhddh}
 \big| \ln \kappa_n \big| = o(n^{1/3}) , \ \ n \rightarrow \infty.
\end{eqnarray}
   Then $[A_n,B_n]$, with
\begin{eqnarray*}
A_n= \hat{\theta}_n - \frac{\Phi^{-1}(1-\kappa_n/2)\, \sigma}{\sqrt{ \Sigma_{k=1}^n X_{k-1}^2}}   \quad  \textrm{\ \ \ \ and \ \ \ \  }
B_n= \hat{\theta}_n  +   \frac{\Phi^{-1}(1-\kappa_n/2)\, \sigma}{\sqrt{ \Sigma_{k=1}^n X_{k-1}^2}}   ,
\end{eqnarray*}
is a $1-\kappa_n$ confidence interval for $\theta$, for $n$ large enough.
\end{corollary}

When the risk probability $\kappa_n$ goes to $0$, we still have the following result.
\begin{corollary}\label{c0kdldss} Assume the condition of Theorem \ref{ththeta2}.   Let $\kappa_n \in (0, 1)$ such that $k_n \rightarrow 0$ and
\begin{eqnarray}\label{3.3sfs}
 \big| \ln \kappa_n \big| =o \big(  n   \big),\ \ \ \ n\rightarrow \infty.
\end{eqnarray}
Then $[A_n,B_n]$, with
\begin{eqnarray*}
A_n= \hat{\theta}_n-\frac{\sigma\sqrt{ 2 |\ln (\kappa_n/2)|}\ }{ \sqrt{ \Sigma_{k=1}^n X_{k-1}^2}  }   \quad   \textrm{and} \ \quad  
B_n=   \hat{\theta}_n +\frac{\sigma\sqrt{ 2 |\ln (\kappa_n/2)|}\ }{ \sqrt{ \Sigma_{k=1}^n X_{k-1}^2}  } , 
\end{eqnarray*}
is a $1-\kappa_n$ confidence interval for $\theta$, for $n$ large enough.
\end{corollary}

\begin{rem}
Following the proof of Theorem  \ref{ththeta2}, by Theorem \ref{th2s},  we can show that  under the conditions of  Theorem  \ref{ththeta2},
the inequality (\ref{gsdsd1})   remains valid when $  (\hat{\theta}_n -\theta  )\sqrt{ \Sigma_{k=1}^n X_{k-1}^2}$ is replaced by
 $  \big(\hat{\theta}_n -\theta \big )\sqrt{\frac{1- \theta^2 }{  n  \sigma^4 }} \sum_{k=1}^n X_{k-1}^2  $ or $  \big( \theta- \hat{\theta}_n\big )\sqrt{\frac{1- \theta^2 }{  n  \sigma^4 }} \sum_{k=1}^n X_{k-1}^2 $.
\end{rem}

\section{Preliminary lemmas}\label{sec4}
\setcounter{equation}{0}

Let  $X=(X_k,\mathcal{F}_k)_{ 0\leq k  \leq n}$
be the  martingale defined by (\ref{xk}). For any real $\lambda$ satisfying $|\lambda| < \epsilon_n^{-1} ,$
we follow the method developed by
Grama  and Haeusler \cite{GH00}, and introduce the following exponential multiplicative martingale $Z(\lambda
)=(Z_k(\lambda ),\mathcal{F}_k)_{0\leq k \leq n},$ where
\[
Z_k(\lambda )=\prod_{i=1}^k\frac{e^{\lambda \xi _i}}{\mathbf{E}(e^{\lambda \xi _i}|
\mathcal{F}_{i-1})},\quad k=1,...,n,\quad Z_0(\lambda )=1.  \label{C-1}
\]
Clearly, for each $k=1,...,n,$   $Z_k(\lambda
) $ defines a probability density on $(\Omega ,\mathcal{F},\mathbf{P}),$   i.e.
\[
 \int Z_{k}(\lambda)  d \mathbf{P} = \mathbf{E}(Z_{k}(\lambda))=1.
 \]
This observation allows us to introduce, for $|\lambda|
 <\epsilon_n^{-1},$ the well-known \emph{conjugate probability measure} $\mathbf{P}_\lambda $ on $(\Omega ,%
\mathcal{F})$ defined by
\begin{equation}
d\mathbf{P}_\lambda =Z_n(\lambda )d\mathbf{P}.  \label{f21}
\end{equation}
Denote by $\mathbf{E}_{\lambda}$ and $\mathbf{E}$ the expectations with respect to $\mathbf{P}_{\lambda}$ and $\mathbf{P}$, respectively.
For all $1\leq i \leq n$, set
\[
\eta_i(\lambda)=\xi_i - b_i(\lambda)\ \ \ \ \ \ \ \ \  \textrm{and} \ \ \ \  \ \ \  \ \  b_i(\lambda)=\mathbf{E}_{\lambda}(\xi_i |\mathcal{F}_{i-1}).  \]
Then we have the following well-known semimartingale decomposition for $X$:
\begin{equation}
X_k=Y_k(\lambda )+B_k(\lambda ),\quad\quad\quad k=1,...,n, \label{xb}
\end{equation}
where
\begin{equation}
Y_k(\lambda )=\sum_{i=1}^k\eta _i(\lambda ) \quad\quad\quad  \textrm{and } \quad\quad\quad  B_k(\lambda )=\sum_{i=1}^kb_i(\lambda )
\end{equation}
are respectively the conjugate martingale and the drift process with respect to $\mathbf{P}_{\lambda}$.

In the proof of Theorem 2.1, we make use of the following three lemmas. The  proofs of the lemmas  are similar to the corresponding assertions in Fan \emph{et al.}\,\cite{FGL13},
therefore we omit the proofs.
\begin{lemma}
\label{l11} Assume that condition (A1) is satisfied. Then
\[
|\mathbf{E}(\xi_{i}^k | \mathcal{F}_{i-1})| \leq   6 k! \epsilon_n^{k}, \ \ \ \textrm{for all} \  k\geq 2,
\]
and
\[
\mathbf{E}(|\xi_{i}|^k | \mathcal{F}_{i-1}) \leq   k! \epsilon_n^{k-2} \mathbf{E}(\xi_{i}^2 | \mathcal{F}_{i-1}), \ \ \ \textrm{for all} \  k\geq 2.
\]
\end{lemma}

The following lemma gives a two-sided bound for the drift process $B_n(\lambda ).$
\begin{lemma}
\label{LEMMA-1-1} Assume that condition (A1) is satisfied.  Then for any constant $\alpha \in (0,1)$ and all $0 \leq \lambda \leq \alpha\, \epsilon_n^{-1} ,$
\begin{eqnarray}\label{f25sa}
( \lambda    -c_\alpha \, \lambda^2 \epsilon_n  ) \langle X \rangle_{n} \leq  B_n(\lambda )  \leq  (\lambda   + c_\alpha \, \lambda^2 \epsilon_n ) \langle X \rangle_{n}.
\end{eqnarray}
\end{lemma}

Next, we introduce the predictable cumulant process $\Psi (\lambda )=(\Psi
_k(\lambda ),\mathcal{F}_k)_{k=0,...,n}$, where
\begin{equation}
\Psi _k ( \lambda )=\sum_{i=1}^k\ln \mathbf{E}\big( e^{\lambda \xi _i} \big|\mathcal{F}_{i-1} \big). \label{f29}
\end{equation}
We have the following  two-sided bound for the predictable cumulant process $\Psi(\lambda ).$
\begin{lemma}
\label{LEMMA-1-2} Assume that condition (A1) is satisfied. Then  for any constant $\alpha \in (0,1)$ and all $0 \leq \lambda \leq \alpha\, \epsilon_n^{-1} ,$
\[
\Big|\Psi _n(\lambda ) - \frac{\lambda ^2}2\left\langle X\right\rangle _n \Big|  \leq  c_\alpha \lambda^3 \epsilon_n \langle X\rangle_n.
\]
\end{lemma}

We show that $\mathbf{P}_\lambda$ has the following property.
\begin{lemma} \label{ghkl}
Assume that conditions (A1) and (A2) are satisfied.  Then for all $0 \leq \lambda =o(\min\{ \epsilon_n^{-1} ,  \delta_n^{-1}\})$, the following two inequalities hold:
for all $2 \leq k \leq 5,$
\begin{eqnarray}
\mathbf{E}_{\lambda}(|\eta_{i}(\lambda) |^{k} | \mathcal{F}_{i-1} )\leq c \, k!\, (2\epsilon_n)^{k-2}\mathbf{E}(\xi_{i}^2 | \mathcal{F}_{i-1} )
\end{eqnarray}
and for all $y\geq 4 \lambda  \delta_n, $
\begin{eqnarray}
\mathbf{P}_\lambda \Big (  |\langle X \rangle_{n} - 1 | \geq y  \Big )
 \leq  C \,  \exp\Big\{  - \frac14 y^2 \, \delta_n^{-2} \Big\}.
\end{eqnarray}
\end{lemma}

\begin{proof}
Notice that   $|\eta_{i}(\lambda) |^k \leq 2^{k-1}(|\xi_{i}|^k+ \mathbf{E}_{\lambda}(|\xi_{i}||\mathcal{F}_{i-1})^k), k \geq 2$. Then for all $k\geq 2$ and all $0\leq \lambda =o(\epsilon_n^{-1} )$, we have
\begin{eqnarray}
\mathbf{E}_{\lambda}\big(|\eta_{i}(\lambda) |^{k} \big| \mathcal{F}_{i-1} \big) &\leq &2^{k-1} \mathbf{E}_{\lambda}\big(|\xi_{i}|^k + \mathbf{E}_{\lambda}(|\xi_{i}| | \mathcal{F}_{i-1})^k \big| \mathcal{F}_{i-1} \big) \nonumber \\
&\leq & 2^k \mathbf{E}_{\lambda}\big(|\xi_{i}|^k  \big| \mathcal{F}_{i-1} \big), \nonumber
\end{eqnarray}
where the last line follows by Jensen's inequality.
Again by Jensen's inequality, it holds $$\mathbf{E}(e^{\lambda\xi_{i}} | \mathcal{F}_{i-1} ) \geq e^{\lambda\mathbf{E}(\xi_{i}| \mathcal{F}_{i-1} )}= 1  .$$  By the last inequality,    Lemma \ref{l11} and condition (A1), it follows that for all $k\geq 2$ and all $0\leq \lambda =o(\epsilon_n^{-1} )$,
\begin{eqnarray}
\mathbf{E}_{\lambda}\big(|\eta_{i}(\lambda) |^{k} \big| \mathcal{F}_{i-1} \big)
&\leq & 2^k \mathbf{E} \big(|\xi_{i}|^k \exp\{|\lambda\xi_{i}|\} \big| \mathcal{F}_{i-1} \big)  = 2^k \sum_{l=0}^{+\infty}  \frac{1}{l!}\lambda^{l } \mathbf{E} ( |\xi_{i}|^{l+k}  \big |  \mathcal{F}_{i-1}   ) \nonumber \\
& \leq& 2^k \sum_{l=0}^{+\infty} (k+l)\cdot\cdot\cdot(l+1) (\lambda \epsilon_n)^{l}  \epsilon_n^{k-2} \mathbf{E} \big(\xi_{i}^2 | \mathcal{F}_{i-1} \big)  \nonumber \\
&\leq & c \,  k!(2\epsilon_n)^{k-2}\mathbf{E} \big(\xi_{i}^2 | \mathcal{F}_{i-1} \big) ,   \label{dfdffs}
\end{eqnarray}
which gives the first desired inequality.
For all $1\leq k \leq n,$ denote $\Delta \left\langle Y(\lambda)   \right\rangle_{k}= \mathbf{E}_{\lambda}\left((\eta_{i}(\lambda) )^2  | \mathcal{F}_{k-1} \right)$ and $\Delta \left\langle
X\right\rangle _k =\mathbf{E} \left( \xi_{k}    | \mathcal{F}_{k-1} \right)$.
 By the definition of  conjugate probability measure, it holds for all $1\leq k \leq n,$
\begin{eqnarray}
\Delta \left\langle Y(\lambda)   \right\rangle_{k}
&= & \frac{\mathbf{E}(\xi
_k^2e^{\lambda \xi _k}|\mathcal{F}_{k-1})}{\mathbf{E}(e^{\lambda \xi _k}|\mathcal{F}%
_{k-1})}-\frac{\mathbf{E}(\xi _ke^{\lambda \xi _k}|\mathcal{F}_{k-1})^2}{\mathbf{E}(e^{\lambda
\xi _k}|\mathcal{F}_{k-1})^2}. \label{f24}
\end{eqnarray}
By  (\ref{f24}) and the inequality $\mathbf{E}(e^{\lambda\xi_{i}} | \mathcal{F}_{i-1} ) \geq 1  $,  it follows that  for all $0 \leq \lambda  =o(\epsilon_n^{-1} ) ,$
\begin{eqnarray}
\left| \Delta \left\langle Y(\lambda)  \right\rangle _k-\Delta \left\langle
X\right\rangle _k\right| &\leq &\left| \frac{\mathbf{E}(\xi _k^2e^{\lambda \xi _k}|%
\mathcal{F}_{k-1})}{\mathbf{E}(e^{\lambda \xi _k}|\mathcal{F}_{k-1})}-\mathbf{E}(\xi _k^2|%
\mathcal{F}_{k-1})\right| +\left| \frac{\mathbf{E}(\xi _k e^{\lambda \xi _k}|\mathcal{F%
}_{k-1})^2}{\mathbf{E}(e^{\lambda \xi _k}|\mathcal{F}_{k-1})^2}\right|  \nonumber \\
&\leq &\left|\mathbf{E}(\xi _k^2e^{\lambda \xi _k}|\mathcal{F}_{k-1})-\mathbf{E}(\xi _k^2|\mathcal{F}_{k-1})\mathbf{E}(e^{\lambda \xi _k}|\mathcal{F}_{k-1})\right|
  + \mathbf{E}(\xi_{k} e^{\lambda \xi _k}|\mathcal{F}_{k-1})^2 . \nonumber
\end{eqnarray}
Using Taylor's expansion for $e^x$, condition (A1) and Lemma
\ref{l11}, we deduce that  for all $0 \leq \lambda  =o(\epsilon_n^{-1} ) ,$
\begin{eqnarray}
\left| \Delta \left\langle Y(\lambda)  \right\rangle _k-\Delta \left\langle
X\right\rangle _k\right|
&\leq & \sum_{l=1}^{\infty}|\mathbf{E}(\xi_{k}^{l+2} | \mathcal{F}_{k-1})|\frac{\lambda^l}{l !} + \Delta \langle X\rangle_{k}
\sum_{l=1}^{\infty}|\mathbf{E}(\xi_{k}^l | \mathcal{F}_{k-1})|\frac{\lambda^l}{l !} \nonumber\\
&& +\bigg(\sum_{l=1}^{\infty}|\mathbf{E}(\xi_{k}^{l+1} | \mathcal{F}_{k-1})|\frac{ \lambda ^l}{l !}\bigg)^2 \nonumber\\
&\leq &  c_0 \lambda\epsilon_n\, \Delta \langle X\rangle_{k} .\label{f56}
\end{eqnarray}
By Lemma \ref{LEMMA-1-2},
it is easy to see that for all $0 \leq  \lambda = o(  \epsilon_n^{-1} ) $ and all $y >0,$
\begin{eqnarray}
&& \mathbf{P}_\lambda \Big (  |\langle X  \rangle_{n} - 1 | \geq y  \Big ) =  \mathbf{E}\Big(Z_n (\lambda) \mathbf{\textbf{1}}_{\{|\langle X\rangle_{n}-1| \geq y   \}} \Big) =\mathbf{E}\Big(\exp\{ \lambda X_n - \Psi _n ( \lambda ) \} \mathbf{\textbf{1}}_{\{|\langle X\rangle_{n}-1| \geq y  \}} \Big)  \nonumber  \\
&&\ \ \ \ \ \leq \ \mathbf{E}\Big(\exp\{ \lambda X_n -  \lambda ^2 \langle X \rangle_n -4 c_\alpha  \lambda^3 \epsilon_n \langle X \rangle_n   +\frac{ \lambda ^2}{2}\langle X \rangle_n +5c_\alpha  \lambda^3 \epsilon_n \langle X \rangle_n  \}  \mathbf{\textbf{1}}_{\{|\langle X\rangle_{n}-1| \geq y  \}} \Big), \nonumber
\end{eqnarray}
where $c_\alpha$ is given by Lemma \ref{LEMMA-1-2}.
Using H\"{o}lder's inequality,   we get for all $0 \leq 2\lambda   = o(  \epsilon_n^{-1} )  $ and all $y>0$,
\begin{eqnarray}
\mathbf{P}_\lambda \Big (  |\langle X \rangle_{n} - 1 | \geq y  \Big )
&\leq&  \bigg[\mathbf{E}\bigg( \exp\Big\{ 2\lambda X_n -  \frac{ (2\lambda) ^2}{2}\langle X \rangle_n - c_\alpha  (2\lambda)^3 \epsilon_n \langle X \rangle_n   \Big \} \bigg)\bigg]^{1/2} \nonumber   \\
&& \times \bigg[\mathbf{E}\bigg( \exp\Big\{  \lambda ^2 \langle X \rangle_n +10c_\alpha  \lambda^3 \epsilon_n \langle X \rangle_n  \Big\} \mathbf{\textbf{1}}_{\{|\langle X\rangle_{n}-1| \geq y   \}}  \bigg)\bigg]^{1/2}  \nonumber .
\end{eqnarray}
By Lemma \ref{LEMMA-1-2} and condition (A2),  we get for all $0 \leq \lambda =o(\min\{ \epsilon_n^{-1} ,  \delta_n^{-1}\})   $ and all $y\geq 4 \lambda  \delta_n, $
\begin{eqnarray}
\mathbf{P}_\lambda \Big (  |\langle X \rangle_{n} - 1 | \geq y  \Big )
&\leq& \bigg[\mathbf{E} \bigg( \exp\Big\{ \lambda ^2 \langle X \rangle_n +10c_\alpha  \lambda^3 \epsilon_n \langle X \rangle_n  \Big\} \mathbf{\textbf{1}}_{\{|\langle X\rangle_{n}-1| \geq y  \}} \bigg) \bigg]^{1/2} \nonumber \\
&\leq& \bigg[ e^{\frac43\lambda ^2}\mathbf{P}  \big (   \langle X \rangle_{n}-1  < -y  \big )  + \mathbf{E} \Big( e^{ \frac43 \lambda ^2 \langle X \rangle_n  } \mathbf{\textbf{1}}_{\{ \langle X\rangle_{n} \geq 1+ y  \}} \Big) \bigg]^{1/2} \nonumber \\
&\leq&  e^{\frac23\lambda ^2}\mathbf{P}  \big (   \langle X \rangle_{n}-1  < -y  \big )^{1/2}+\,   \, e^{\frac{2}{3} \lambda ^2 (1+y ) } \Big[ \mathbf{P}( \langle X\rangle_{n}-1 \geq y  )\Big]^{1/2}  \nonumber \\
& &   + \bigg[\frac43 \lambda^2\int_{1+y}^\infty  e^{\frac43 \lambda^2 t -  (t-1)^2 \, \delta_n^{-2} }   dt \bigg]^{1/2} \nonumber \\
&\leq& C \, e^{\lambda^2  (1+y )-\frac12 y^2 \, \delta_n^{-2} }   \nonumber \\
&\leq&  C \, \exp\Big\{  -\frac1{4} y^2 \, \delta_n^{-2} \Big\}.\label{fsdghn}
\end{eqnarray}
This completes the proof of Lemma  \ref{ghkl}.
\end{proof}

In the next lemma, we establish a rate of convergence in the central limit
theorem, usually termed as Berry-Esseen's bound, for the conjugate martingale $Y(\lambda)=(Y_k(\lambda),\mathcal{F}_k)_{k=1,...,n} $ under the probability measure $\mathbf{P}_{ \lambda  },$
where $Y_k(\lambda)=\sum_{i=1}^k\eta _i(\lambda) .$ For $1\leq k \leq n,$ denote by $\left\langle Y(\lambda) \right\rangle_{k}= \sum_{i\leq k}\mathbf{E}_{\lambda}(
(\eta_{i}(\lambda))^2 | \mathcal{F}_{i-1})$ the quadratic characteristic of the conjugate
martingale $Y(\lambda)$.

\begin{lemma}
\label{LEMMA4}
Assume that conditions (A1) and (A2) are satisfied.  Then  for  all  $0 \leq \lambda =o(\min\{ \epsilon_n^{-1} ,  \delta_n^{-1}\})$,
\begin{eqnarray}\label{fdsfdd}
&&\sup_{x \in \mathbf{R}}\Big| \mathbf{P}_\lambda \big(\, Y_n(\lambda) \leq x  \big )-\Phi (x)\Big|   \leq
c \, \Big(  \lambda ( \epsilon_n + \delta_n) +\delta_n|\ln \delta_n|+\epsilon_n|\ln \epsilon_n| \Big)
\end{eqnarray}
and
\begin{eqnarray}\label{fsfcsgf}
&&\sup_{x \in \mathbf{R}}\bigg| \mathbf{P}_\lambda \Big(\, Y_n(\lambda) \leq x, \   | \langle X\rangle_n-1|  \ \leq \   c_0( \lambda \delta_n   +\delta_n\sqrt{|\ln \delta_n|})   \Big )-\Phi (x)\bigg|   \nonumber  \\
&&\ \ \ \ \ \ \ \ \ \ \ \ \ \ \ \ \ \ \ \ \ \ \  \  \ \  \ \  \ \ \ \ \ \ \ \ \ \ \ \ \ \ \ \ \ \ \ \    \ \leq \   c\,  \Big(  \lambda ( \epsilon_n + \delta_n) +\delta_n|\ln \delta_n|+\epsilon_n|\ln \epsilon_n| \Big) ,
\end{eqnarray}
with any $c_0$  large enough.
\end{lemma}

Grama and  Haeusler \cite{GH00} (cf.\ Lemma 3.3 therein) obtained
 a similar bound  with the conditions  $|\xi_{i}|\leq   \epsilon_n$ and $\| \left\langle X\right\rangle _n-1\|_\infty \leq \delta_n^2 $,
 which is a particular case of conditions (A1) and (A2).
Thus Lemma \ref{LEMMA4} is an extension  of Lemma 3.3 in Grama and  Haeusler \cite{GH00}.

\noindent

\begin{proof}
Clearly, it holds
\begin{eqnarray}
&&\sup_{x\in \mathbf{R} }\Big| \mathbf{P}_\lambda \Big(\, Y_n(\lambda)  \leq x, | \langle X\rangle_n-1| \leq   c_0( \lambda \delta_n   +\delta_n\sqrt{|\ln \delta_n|})   \Big )-\Phi (x)\Big| \nonumber \\
&&\ \ \ \ \ \ \ \ \ \   \leq \sup_{x\in \mathbf{R}}\Big| \mathbf{P}_\lambda (\, Y_n(\lambda) \leq x   )-\Phi (x)\Big|     + \ \mathbf{P}_\lambda \Big(  | \langle X\rangle_n-1| >  c_0( \lambda \delta_n   +\delta_n\sqrt{|\ln \delta_n|})   \Big). \label{inndsfg}
\end{eqnarray}
By Lemma  \ref{ghkl}, it is easy to see that for  all  $0 \leq \lambda =o(\min\{ \epsilon_n^{-1} ,  \delta_n^{-1}\})$,
\begin{eqnarray}
\mathbf{P}_\lambda \Big (  | \langle X\rangle_n-1| >  c_0( \lambda \delta_n   +\delta_n\sqrt{|\ln \delta_n|})   \Big )
&\leq&  c\, \exp\Big\{  -\frac{c_0^2}{4}   (  \delta_n\sqrt{|\ln \delta_n|} )^2 \, \delta_n^{-2} \Big\}  \nonumber  \\
&\leq & c \, \delta_n, \label{fsfddghn}
\end{eqnarray}
for any $c_0$  large enough.
Inequality (\ref{fsfcsgf}) is a simple consequence of (\ref{fdsfdd}), (\ref{inndsfg}) and (\ref{fsfddghn}).
Thus, we only need to prove (\ref{fdsfdd}). The remanning proof of the lemma is much more complicated and we give details in the
supplemental article Fan and Shao \cite{FS22}.
\end{proof}

If $\lambda=0$, then
$Y_n(\lambda )=X_{n}$ and $\mathbf{P}_\lambda=\mathbf{P}.$
So Lemma \ref{LEMMA4} implies the following Berry-Esseen  bound.
\begin{thm}\label{th4}
Assume that conditions (A1) and (A2) are satisfied.  Then the following inequality holds
\begin{equation}  \label{f30}
\sup_{x \in \mathbf{R}} \Big|\mathbf{P}(X_n \leq x   )-\Phi \left( x\right) \Big| \leq c_{  p} \Big(  \epsilon_n|\ln \epsilon_n| +\delta_n|\ln \delta_n| \Big).
\end{equation}
\end{thm}

It is known that under the conditions $\| \xi_i\|_\infty\leq \epsilon_n$ and $\| \langle X\rangle_n-1\|_\infty \leq \delta_n^2 $,  the best possible convergence rate of Berry-Esseen's  bound for martingales  is in order of $ \epsilon_n |\ln \epsilon_n| +\delta_n$,
see  Bolthausen \cite{Bo82} when $\epsilon_n=O(1/\sqrt{n})$ and   \cite{F19} for general $\epsilon_n$.
Thus, the  term  $ \epsilon_n |\ln \epsilon_n|$ is  the best possible.

\section{\textbf{Proof of Theorem  \ref{th0}}\label{sec3.1}}
\setcounter{equation}{0}

The first assertion of Theorem \ref{th0} will be deduced by the combination of the following two lemmas (Lemmas \ref{lem41} and \ref{lem42}), which are stated
and proved respectively in this section.  The second assertion of Theorem \ref{th0} follows from  the first one applied to $(-X_i, \mathcal{F}_i)_{1\leq i \leq n}$.
 The proofs of Lemmas \ref{lem41} and \ref{lem42} are close  to the proofs  of Theorems 2.1 and 2.2  of  Fan \emph{et al.}\ \cite{FGL13}.
However, Fan \emph{et al.}\ \cite{FGL13}  considered the standardized martingales $X_n $ with $ \left\| \left\langle X\right\rangle _n-1\right\|_\infty \leq \delta_n^2 $  instead of the normalized martingales $X_n /\sqrt{ \langle X\rangle_n }$ with condition (A2).

\subsection{Upper bound for normalized martingales}
The following lemma gives an upper bound for the relative error of normal approximation.
\begin{lemma} \label{lem41}
Assume that conditions (A1) and (A2) are satisfied.   Then for  all $0 \leq  x    =o( \min\{\epsilon_n^{-1},   \delta_n^{-1 }   \}),$
\begin{eqnarray}
 \ln \frac{\mathbf{P}(X_n >x\sqrt{ \langle X\rangle_n }\,)}{1-\Phi \left( x\right)}  \ \leq \ c \, \bigg(x^3 (\epsilon_n  +    \delta_n) +      (1+ x) ( \delta_n|\ln \delta_n| + \epsilon_n|\ln \epsilon_n|)         \bigg)  ,
\end{eqnarray}
where   $c $ does not depend on $(\xi _i,\mathcal{F}_i)_{i=0,...,n}$, $n$ and $x$.
\end{lemma}

\begin{proof}
 For the sake of simplicity of notations, denote
$$ \delta_n(\lambda)= c_0  ( \lambda  \delta_n  +\delta_n\sqrt{|\ln \delta_n|} )   ,$$
where  $c_0$ is a constant large enough.
According to the change of probability measure  (\ref{f21}),  we deduce that for all $0\leq \lambda =o( \epsilon_n^{-1}),$
\begin{eqnarray}
&& \mathbf{P}\Big(X_n >x\sqrt{ \langle X\rangle_n },  \left| \left\langle X\right\rangle _n-1\right|  \leq \delta_n(\lambda)  \Big) \nonumber  \\
&& = \mathbf{E}_\lambda \left( Z_n (\lambda)^{-1}\mathbf{1}_{\{ X_n>x\sqrt{ \langle X\rangle_n } , \ \left| \left\langle X\right\rangle _n-1\right| \leq \delta_n(\lambda)  \}} \right) \nonumber\\
&&\leq \mathbf{E}_\lambda \left(e^{-\lambda X_n+\Psi _n(\lambda ) } \mathbf{1}_{\{ X_n  >x \sqrt{1-\delta_n(\lambda)}, \ \left| \left\langle X\right\rangle _n-1\right|  \leq \delta_n(\lambda)  \}} \right) \nonumber\\
&& = \mathbf{E}_\lambda \left( e^{
-\lambda Y_n(\lambda)-\lambda B_{n}(\lambda)+\Psi _n(\lambda)}  \mathbf{1}_{\{Y_n(\lambda)+B_{n}(\lambda)>x \sqrt{1-\delta_n(\lambda)}   , \ \left| \left\langle X\right\rangle _n-1\right|  \leq \delta_n(\lambda)  \}} \right) \nonumber \\
&& \leq \mathbf{E}_\lambda \left( e^{-\lambda Y_n(\lambda)-\frac{\lambda^2}{2} \left\langle X\right\rangle _n  + c_0 \lambda^3 \epsilon_n \left\langle X\right\rangle _n} \mathbf{1}_{\{Y_n(\lambda)+B_{n}(\lambda)>x \sqrt{1-\delta_n(\lambda)}   , \ \left| \left\langle X\right\rangle _n-1\right|  \leq \delta_n(\lambda)  \}} \right),
\label{f32}
\end{eqnarray}
where the last line follows by Lemmas \ref{LEMMA-1-1} and \ref{LEMMA-1-2}.
Notice that $B_{n}(\lambda)\leq  (\lambda    +c_\alpha \, \lambda^2 \epsilon_n) \langle X \rangle_{n} $  (cf.  Lemma
\ref{LEMMA-1-1}), where $c_\alpha$ is given by inequality (\ref{f25sa}). For   $0 \leq \lambda  =o( \epsilon_n^{-1}) ,$ let
$\overline{\lambda}=\overline{\lambda}(x)$ be the positive
solution of the equation
\begin{equation}\label{f33}
\Big(\lambda  + c_\alpha \, \lambda^2 \epsilon_n   \Big)\Big(1+\delta_n (\lambda) \Big)  =x \sqrt{1-\delta_n (\lambda)} .
\end{equation}
The definition of $\overline{\lambda}$ implies that there exist $c_{1}, c_{ 2} >0$ such that  for all $0 \leq x  =o( \min\{\epsilon_n^{-1},   \delta_n ^{-1 } \})$,
\begin{equation}\label{f34}
 c_{1}x \leq \overline{\lambda}  \leq  x
\end{equation}
and
\begin{equation}\label{f35}
\overline{\lambda}=x - c_{2} \theta(x) \big( x^2  ( \epsilon_n  +  \delta_n) + x  \delta_n\sqrt{|\ln \delta_n|} \big) \in [0, \,  o(\min\{\epsilon_n^{-1},   \delta_n^{-1 }   \})) ,
\end{equation}
where $0\leq \theta(x) \leq 1.$
 From (\ref{f32}), using   equality (\ref{f33}),
 we deduce that   for all $0 \leq x =o( \min\{\epsilon_n^{-1},   \delta_n^{-1 } \}) ,$
\begin{eqnarray}\label{gfdgf}
&&\mathbf{P}\Big(X_n >x \sqrt{ \langle X\rangle_n } ,  \left| \left\langle X\right\rangle _n-1\right|  \leq \delta_n(\overline{\lambda}) \Big ) \nonumber \\
 &&\leq \exp\Big\{ c_{3} \,\Big(\overline{\lambda}^{3}(\epsilon_n + \delta_n) + \overline{\lambda}^2  \delta_n\sqrt{|\ln \delta_n|} \Big ) -\frac12\overline{\lambda}^2 \Big\}\mathbf{E}_{\overline{\lambda}}\left(e^{-%
\overline{\lambda}Y_n(\overline{\lambda})}\mathbf{1}_{\{ Y_n(\overline{\lambda})>0 ,  \ \left| \left\langle X\right\rangle _n-1\right|  \leq \delta_n(\overline{\lambda}) \}} \right)\!\! .
 \end{eqnarray}
Clearly, it holds for all $0 \leq x =o( \min\{\epsilon_n^{-1},   \delta_n^{-1 } \}) ,$
\begin{eqnarray}\label{f37d1}
&& \mathbf{E}_{\overline{\lambda}} \left( e^{-%
\overline{\lambda}Y_n(\overline{\lambda})}\mathbf{1}_{\{ Y_n(\overline{\lambda})>0,   \ \left| \left\langle X\right\rangle _n-1\right|  \leq \delta_n (\overline{\lambda})\}}\right)
\nonumber \\
&& \ \ \ \ \ \ \  \ \ \ \ \  \ \ \ \ \ = \int_{0}^{\infty} \overline{\lambda} e^{-\overline{\lambda} y}  \mathbf{P}_{\overline{\lambda}}\Big(0 < Y_n(\overline{\lambda})\leq y,    \ \left| \left\langle X\right\rangle _n-1\right|  \leq \delta_n (\overline{\lambda})  \Big) dy.
\end{eqnarray}
Similarly, for a standard normal random variable $\mathcal{N}$, it holds
\begin{eqnarray}\label{f37d2}
  \mathbf{E} \left( e^{-\overline{\lambda}\mathcal{N}}\mathbf{1}_{\{ \mathcal{N}>0\}} \right)  = \int_{0}^{\infty} \overline{\lambda} e^{-\overline{\lambda} y}   \mathbf{P} (0 < \mathcal{N} \leq y  ) dy.
\end{eqnarray}
From (\ref{f37d1}) and (\ref{f37d2}), it follows that for all $0 \leq x =o( \min\{\epsilon_n^{-1},   \delta_n^{-1 } \}) ,$
\begin{eqnarray}
&&\left|\mathbf{E}_{\overline{\lambda}}\left(e^{-%
\overline{\lambda}Y_n(\overline{\lambda})}\mathbf{1}_{\{ Y_n(\overline{\lambda})>0,    \ \left| \left\langle X\right\rangle _n-1\right|  \leq \delta_n (\overline{\lambda})\}}\right)- \mathbf{E} \left( e^{-%
\overline{\lambda}\mathcal{N}}\mathbf{1}_{\{ \mathcal{N}>0\}}\right) \right| \nonumber \\
&&\ \ \ \ \ \ \ \ \ \ \ \ \ \ \ \ \ \ \ \ \ \ \ \  \ \  \leq 2\sup_y \bigg| \mathbf{P}_{\overline{\lambda}} \Big(Y_n(\overline{\lambda} )\leq y,    \ \left| \left\langle X\right\rangle _n-1\right|  \leq \delta_n (\overline{\lambda})\Big)-\Phi (y) \bigg|.\nonumber
\end{eqnarray}
Using Lemma \ref{LEMMA4}, we have the following bound  for all $0 \leq x =o( \min\{\epsilon_n^{-1},   \delta_n^{-1 } \}) ,$
\begin{eqnarray}
&& \bigg|\mathbf{E}_{\overline{\lambda}}\Big(e^{-%
\overline{\lambda}Y_n(\overline{\lambda})}\mathbf{1}_{\{ Y_n(\overline{\lambda})>0, \ \left| \left\langle X\right\rangle _n-1\right|  \leq \delta_n (\overline{\lambda})\}}\Big)- \mathbf{E}\left( e^{-%
\overline{\lambda}\mathcal{N}}\mathbf{1}_{\{ \mathcal{N}>0\}}\right)\bigg|  \nonumber \\
 && \ \ \ \ \ \ \ \ \ \ \ \ \ \ \ \ \ \  \ \ \ \ \  \ \ \ \ \ \ \ \ \ \ \ \ \ \ \ \   \  \leq  c_{3 }  \bigg(   \overline{\lambda}  ( \epsilon_n+ \delta_n)   +\delta_n|\ln \delta_n| + \epsilon_n|\ln \epsilon_n| \bigg)  . \label{f38}
\end{eqnarray}
Combining  (\ref{gfdgf}) and (\ref{f38}) together, we get for all $0 \leq x =o( \min\{\epsilon_n^{-1},   \delta_n ^{-1 } \}) ,$
\begin{eqnarray}
&&  \mathbf{P}\bigg(X_n >x \sqrt{ \langle X\rangle_n },    \ \left| \left\langle X\right\rangle _n-1\right|  \leq  \delta_n (\overline{\lambda})  \bigg)
\leq \exp\bigg\{ c_{4} \,\Big( \overline{\lambda}^3   (\epsilon_n  +  \delta_n) +  \overline{\lambda}^2  \delta_n\sqrt{|\ln \delta_n|} \Big ) -\frac12\overline{\lambda}^2 \bigg\} \nonumber \\
 &&\ \ \ \ \ \ \ \ \ \ \ \ \ \ \ \  \  \ \  \ \ \ \ \ \ \ \ \ \ \ \ \times \bigg( \mathbf{E}  \Big( e^{- \overline{\lambda}\mathcal{N}}\mathbf{1}_{\{ \mathcal{N}>0\}}  \Big)+  c_{3}  \Big(   \overline{\lambda}  ( \epsilon_n+ \delta_n)   +\delta_n|\ln \delta_n| + \epsilon_n|\ln \epsilon_n| \Big)  \bigg). \label{ineq510}
\end{eqnarray}
Since
\begin{equation} \label{fdgsgh}
e^{- \lambda^2/2}\mathbf{E} \left( e^{-
 \lambda\mathcal{N}}\mathbf{1}_{\{ \mathcal{N}>0\}} \right) =\frac{1}{\sqrt{2\pi}}\int_0^{\infty}e^{-(y+\lambda)^2/2}
  dy=1-\Phi \left(  \lambda\right)
\end{equation}
and
\begin{equation}
\frac 1{\sqrt{  \pi}(1+ \lambda)}\ e^{- \lambda^2/2} \geq 1-\Phi \left(  \lambda\right)
\geq
\frac 1{\sqrt{2 \pi}(1+ \lambda)}\ e^{- \lambda^2/2},\ \ \ \ \lambda\geq 0
\label{f39}
\end{equation}
(see \cite{GH00}), we deduce that
for all $0\leq x =o( \min\{\epsilon_n^{-1},   \delta_n^{-1 } \}) ,$
\begin{eqnarray}
&&\frac{ \mathbf{P}\Big(X_n >x\sqrt{ \langle X\rangle_n } ,    \ \left| \left\langle X\right\rangle _n-1\right|  \leq \delta_n (\overline{\lambda}) \Big)}{1-\Phi \left( \overline{\lambda}\right)}
\leq \exp\bigg\{ c_{4} \,\Big( \overline{\lambda}^3  (\epsilon_n   +    \delta_n) +  \overline{\lambda}^2  \delta_n\sqrt{|\ln \delta_n|}  \Big)  \bigg\} \nonumber \\
 &&\quad\quad \quad\quad   \quad\quad \ \ \quad\quad\quad\quad \quad\quad    \times\bigg(\, 1+ c_{5} (1+ \overline{\lambda} )  \Big(   \overline{\lambda}  ( \epsilon_n+ \delta_n)   +\delta_n|\ln \delta_n| + \epsilon_n|\ln \epsilon_n| \Big)     \bigg).  \label{f40}
\end{eqnarray}
Next, we would like to compare $1-\Phi (\overline{\lambda})$ with $1-\Phi (x)$.
  By (\ref{f34}), (\ref{f35}) and (\ref{f39}),  we deduce that for all $0 \leq x =o( \min\{\epsilon_n^{-1},   \delta_n^{-1 } \}) ,$
\begin{eqnarray}
   1 &\geq& \frac{\int_{\overline{\lambda}}^{ \infty}\exp\{- t^2/2 \}d t}{\int_{x}^{ \infty}\exp\{- t^2/2 \}d t}\geq
  1+\frac{\int_{\overline{\lambda}}^{x}\exp\{ -t^2/2 \} d t}{\int_{x}^{ \infty}\exp\{-t^2/2\}d t}\nonumber\\
   & \geq & 1+c_{1}(1+x)(x-\overline{\lambda}) \exp\Big\{ \frac12(x^2-\overline{\lambda}^2)  \Big\}\nonumber\\
   & \geq & \exp\bigg\{ c_{2}\, (1+x)\Big( x^2 (\epsilon_n  +   \delta_n) + x  \delta_n\sqrt{|\ln \delta_n|} \Big) \bigg\}.\label{f41}
\end{eqnarray}
So, it holds for all $0 \leq x =o( \min\{\epsilon_n^{-1},   \delta_n^{-1 } \}) ,$
\begin{equation}\label{f42}
1-\Phi \left( \overline{\lambda}\right) =\Big( 1-\Phi (x)\Big)\exp \bigg\{  \theta_{1}  c_{6}\, (1+x)\Big(
x^2 (\epsilon_n +    \delta_n)+  x\delta_n\sqrt{|\ln \delta_n|} \Big) \bigg\},
\end{equation}
where $0\leq  \theta_{1}  \leq 1.$
Implementing (\ref{f42}) in (\ref{f40}), by  (\ref{f34}), we obtain  for all $0 \leq  x   =o( \min\{\epsilon_n^{-1},   \delta_n^{-1 }   \}), $
\begin{eqnarray*}
&&\frac{\mathbf{P}\Big(X_n  >x\sqrt{ \langle X\rangle_n } ,    \ \left| \left\langle X\right\rangle _n-1\right|  \leq \delta_n (\overline{\lambda}) \Big)}{1-\Phi \left( x\right) }\\
&&\leq \exp\bigg\{ c_{7} \Big( (1+x)x^2(\epsilon_n +    \delta_n) +  (1+x +x^2)   \delta_n\sqrt{|\ln \delta_n|}    \Big ) \bigg \} \nonumber \\
&& \ \ \ \   \times \bigg(\, 1+ c_{5 }\, (1+x) \Big(   x  ( \epsilon_n+ \delta_n)   +\delta_n|\ln \delta_n| + \epsilon_n|\ln \epsilon_n| \Big)  \bigg)\\
&&\leq \exp\bigg\{ c_{8} \Big(x^3( \epsilon_n  +     \delta_n) + (x+x^2)  \delta_n\sqrt{|\ln \delta_n|} + (1+x)  x  ( \epsilon_n+ \delta_n)+ (1+   x)  ( \delta_n|\ln \delta_n| + \epsilon_n|\ln \epsilon_n|)    \Big)  \bigg\}  \nonumber  ,
\end{eqnarray*}
where the last line follow by the inequality $ 1+x  \leq e^x$ for all $x\geq 0.$ Clearly,   for all $0\leq x \leq \sqrt{|\ln \delta_n|},$
\begin{eqnarray*}
(x+x^2)  \delta_n\sqrt{|\ln \delta_n|} + (1+x)  x  ( \epsilon_n+ \delta_n)  &\leq& (1+x)  \delta_n|\ln \delta_n|  + (1+x)( \delta_n|\ln \delta_n| + \epsilon_n|\ln \epsilon_n|)(\ln 2)^{-1} \\
   &\leq& 3 (1+x)( \delta_n|\ln \delta_n| + \epsilon_n|\ln \epsilon_n|)
\end{eqnarray*}
and for all $x>\sqrt{|\ln \delta_n|},$
\begin{eqnarray*}
(x+x^2) \delta_n\sqrt{|\ln \delta_n|} + (1+x)  x  ( \epsilon_n+ \delta_n)  &\leq& 2 x^3 \delta_n   +  (1+x)( \delta_n|\ln \delta_n| + \epsilon_n|\ln \epsilon_n|)(\ln 2)^{-1} \\
   &\leq& 2\Big( x^3( \epsilon_n  +     \delta_n) + (1+x)( \delta_n|\ln \delta_n| + \epsilon_n|\ln \epsilon_n|) \Big).
\end{eqnarray*}
Thus for all $0 \leq  x   =o( \min\{\epsilon_n^{-1},   \delta_n^{-1 }   \}), $
 \begin{eqnarray*}
 \frac{\mathbf{P}\Big(X_n  >x\sqrt{ \langle X\rangle_n } ,    \ \left| \left\langle X\right\rangle _n-1\right|  \leq \delta_n (\overline{\lambda}) \Big)}{1-\Phi \left( x\right) }\leq \exp\Bigg\{ c_{9} \bigg(x^3( \epsilon_n  +     \delta_n)  + (1+   x)  ( \delta_n|\ln \delta_n| + \epsilon_n|\ln \epsilon_n|)    \bigg)  \Bigg\}  \nonumber .
\end{eqnarray*}
By condition (A2),  we have for all $0 \leq  x   =o( \min\{\epsilon_n^{-1},   \delta_n^{-1 }   \}), $
$$ \mathbf{P}\Big( \left| \left\langle X\right\rangle _n-1\right|  > \delta_n (\overline{\lambda})  \Big)  \leq  C \exp\Big\{  -(\delta_n (\overline{\lambda}))^2 \delta_n^{-2} \Big\}.$$
Taking $c_0$ large enough, by (\ref{f39}), we deduce  that   for  all $0 \leq  x   =o( \min\{\epsilon_n^{-1},   \delta_n^{-1 }   \}), $
 \begin{eqnarray*}
 \frac{\mathbf{P}\Big(  \left| \left\langle X\right\rangle _n-1\right| > \delta_n (\overline{\lambda}) \Big)}{1-\Phi \left( x\right) }  \leq
  c\,(1+x) \exp\Bigg\{ \frac{x^2}{2} - (\delta_n (\overline{\lambda}))^2 \delta_n^{-2}   \Bigg\}  \leq   \delta_n .
\end{eqnarray*}
Notice that
\begin{eqnarray*}
 \mathbf{P}\Big(X_n >x \sqrt{ \langle X\rangle_n }\Big) \leq \mathbf{P}\Big(X_n >x\sqrt{ \langle X\rangle_n } ,    \ \left| \left\langle X\right\rangle _n-1\right|  \leq \delta_n (\overline{\lambda}) \Big)  + \ \mathbf{P}\Big(   \left| \left\langle X\right\rangle _n-1\right|  > \delta_n (\overline{\lambda}) \Big) .
\end{eqnarray*}
Hence, we get for all $0\leq  x   = o( \min\{\epsilon_n^{-1},   \delta_n ^{-1 }   \}),$
 \begin{eqnarray*}
\frac{\mathbf{P}\Big(X_n>x  \sqrt{ \langle X\rangle_n }  \Big)}{1-\Phi \left( x\right) } &\leq& \frac{\mathbf{P}\Big(X_n >x \sqrt{ \langle X\rangle_n } ,    \ \left| \left\langle X\right\rangle _n-1\right|  \leq \delta_n (\overline{\lambda}) \Big)}{1-\Phi \left( x\right) }  +  \frac{\mathbf{P}\Big(  \left| \left\langle X\right\rangle _n-1\right|  >  \delta_n (\overline{\lambda}) \Big)}{1-\Phi \left( x\right) }\\
&\leq & \exp\Bigg\{  c_{1} \bigg(x^3 (\epsilon_n  +    \delta_n) + (1+ x)  ( \delta_n|\ln \delta_n| + \epsilon_n|\ln \epsilon_n|)       \bigg)  \Bigg \}
 +  \delta_n   \\
&\leq & \exp\Bigg\{  c_{2} \bigg(x^3 (\epsilon_n  +    \delta_n) + (1+x)  ( \delta_n|\ln \delta_n| + \epsilon_n|\ln \epsilon_n|)      \bigg) \Bigg \},
\end{eqnarray*}
which implies the desired inequality.
This completes the proof  of   Lemma \ref{lem41}.
\end{proof}

\subsection{Lower bound for normalized martingales}
The next lemma gives a lower bound for the relative error of normal approximation.
\begin{lemma} \label{lem42}
Assume that conditions (A1) and (A2) are satisfied.   Then   for all $0 \leq  x   =  o( \min\{\epsilon_n^{-1},   \delta_n^{-1}  \}),$
\begin{equation}
 \ln \frac{\mathbf{P}(X_n >x\sqrt{ \langle X\rangle_n })}{1-\Phi \left( x\right)}  \geq- c \,  \bigg( x^3  (\epsilon_n  + \delta_n)+   (1+  x )  \big(  \delta_n|\ln \delta_n| + \epsilon_n|\ln \epsilon_n| \big) \bigg),
\end{equation}
where   $c $ does not depend on $(\xi _i,\mathcal{F}_i)_{i=0,...,n}$, $n$ and $x$.
\end{lemma}

\begin{proof}
Recall
\[
\delta_n(\lambda)=c_0  ( \lambda  \delta_n  +\delta_n\sqrt{|\ln \delta_n|} ) ,
\]
where  $c_0$ is positive constant large enough.    By an argument similar to (\ref{f32}), we have for all $0\leq \lambda =o( \epsilon_n^{-1}),$
\begin{eqnarray}
&& \mathbf{P}\Big(X_n >x \sqrt{ \langle X\rangle_n },  \left| \left\langle X\right\rangle _n-1\right|  \leq \delta_n(\lambda)  \Big) \nonumber  \\
&& \ \ \ \ \ \  \geq \mathbf{E}_\lambda \left( e^{-\lambda Y_n(\lambda)-\frac{\lambda^2}{2} \left\langle X\right\rangle _n  -c_0 \lambda^3 \epsilon_n \left\langle X\right\rangle _n }  \mathbf{1}_{\{Y_n(\lambda)+B_{n}(\lambda)>x \sqrt{1+\delta_n(\lambda)}   , \ \left| \left\langle X\right\rangle _n-1\right|  \leq \delta_n(\lambda)  \}} \right).  \label{f32d}
\end{eqnarray}
Notice that $B_{n}(\lambda) \leq (\lambda  -c_\alpha \, \lambda^2 \epsilon_n ) \langle X \rangle_{n} $  (cf.  Lemma
\ref{LEMMA-1-1}), where $c_\alpha$ is given by inequality (\ref{f25sa}). Let
$\underline{\lambda}=\underline{\lambda}(x)$ be the smallest solution of the equation
\begin{equation}\label{f44}
\Big(\lambda  -c_\alpha \, \lambda^2 \epsilon_n   \Big)\Big(1-\delta_n(\lambda) \Big)  =x \sqrt{1+\delta_n(\lambda)}.
\end{equation}
The definition of $\underline{\lambda}$ implies that    for all  $0 \leq x  =o( \min\{\epsilon_n^{-1},   \delta_n^{-1}  \})$,
\begin{equation}\label{f45}
x\leq \underline{\lambda}   \leq c_1 \,x
\end{equation}
and
\begin{equation}\label{f46}
\underline{\lambda}=x - c_{2} \theta(x) \big(  x^2 (\epsilon_n+ \delta_n)  + x\delta_n\sqrt{|\ln \delta_n|}  \big ) \in [0, \, o( \min\{\epsilon_n^{-1},   \delta_n^{-1}  \}\,] ,
\end{equation}
where $0\leq \theta(x) \leq 1$.
 From (\ref{f32d}), using Lemmas \ref{LEMMA-1-1},  \ref{LEMMA-1-2} and equality (\ref{f44}),
  we get  for all   $0 \leq x =o( \min\{\epsilon_n^{-1},  \delta_n^{-1}  \}),$
\begin{eqnarray}\label{jknjssa}
&& \mathbf{P}\Big(X_n > x \sqrt{ \langle X\rangle_n },  \left| \left\langle X\right\rangle_n-1\right|  \leq \delta_n(\underline{\lambda})  \Big) \geq \exp\bigg\{ -c_{3} \Big( \underline{\lambda}^3 (\epsilon_n+\delta_n)  + \underline{\lambda}^2 \delta_n \sqrt{ |\ln \delta_n|} \Big) -\frac12\underline{\lambda}^2 \bigg\} \nonumber  \\
&& \ \  \quad\quad \quad\quad\quad\quad \quad\quad \quad\quad \quad\quad \quad\quad  \quad\quad \quad \ \ \times \,\mathbf{E}_{\underline{\lambda}} \left(e^{- \underline{\lambda}Y_n(\underline{\lambda})}\mathbf{1}_{\{ Y_n(\underline{\lambda})>0, \   | \left\langle X\right\rangle_n-1 |  \leq \delta_n(\underline{\lambda}) \}} \right).
\end{eqnarray}
We distinguish two cases to estimate the right hand side of the last inequality.

First, we consider the case    $0\leq \underline{\lambda}\leq \alpha_1 \min \{\epsilon_n^{-1/2}, \delta_n^{- 1/2} \}$, where $\alpha_1 >0$ is a small   constant
whose exact value will be given later.
The argument for  the proof of inequality (\ref{ineq510}) holds also when $\overline{\lambda}$ is replace by $\underline{\lambda}$,
thus  we  have
\begin{eqnarray*}
 && \mathbf{P}\Big(X_n >x  \sqrt{ \langle X\rangle_n },    \ \left| \left\langle X\right\rangle _n-1\right|  \leq   \delta_n(\underline{\lambda})  \Big)
\ \geq \ \exp\bigg\{ -c_{3} \Big( \underline{\lambda}^3 (\epsilon_n+\delta_n)  + \underline{\lambda}^2 \delta_n \sqrt{ |\ln \delta_n|} \Big) -\frac12\underline{\lambda}^2 \bigg\} \\ \ && \ \ \ \ \ \ \ \ \ \ \ \ \ \ \ \ \ \ \ \    \ \ \ \ \ \ \ \ \   \ \ \ \ \ \ \ \ \ \ \ \   \times\bigg( \mathbf{E}  \Big( e^{- \underline{\lambda}\mathcal{N}}\mathbf{1}_{\{ \mathcal{N}>0\}}  \Big)-    c_{4}  \Big(   \underline{\lambda}  ( \epsilon_n+ \delta_n)   +\delta_n|\ln \delta_n| + \epsilon_n|\ln \epsilon_n| \Big)  \bigg).
\end{eqnarray*}
Using the inequalities (\ref{fdgsgh}) and (\ref{f39}),
we obtain the following lower bound on tail probabilities:
\begin{eqnarray}
 && \frac{\mathbf{P}\Big(X_n >x \sqrt{ \langle X\rangle_n } ,    \ \left| \left\langle X\right\rangle _n-1\right|  \leq  \delta_n(\underline{\lambda})  \Big)}{1-\Phi \left( \underline{\lambda}\right)}     \geq \exp\bigg\{ - c_{3} \Big( \underline{\lambda}^3 (\epsilon_n+\delta_n)  + \underline{\lambda}^2 \delta_n \sqrt{ |\ln \delta_n|} \Big)  \bigg\} \nonumber \\
  &&\  \ \ \ \ \ \ \ \ \ \ \ \ \ \ \ \ \ \ \ \  \ \  \ \ \ \   \ \ \ \ \ \  \ \ \ \ \  \times\bigg( 1-    c_{4}(1+  \underline{\lambda} )  \Big(   \underline{\lambda}  ( \epsilon_n+ \delta_n)   +\delta_n|\ln \delta_n| + \epsilon_n|\ln \epsilon_n| \Big)  \bigg).  \label{f51}
\end{eqnarray}
Taking  $\alpha_1 =  (4 c_{4} )^{-1/2}$,
we deduce that  for all $0\leq \underline{\lambda} \leq \alpha_1 \min \{\epsilon_n^{-1/2}, \delta_n^{- 1/2} \} $,
\begin{eqnarray}\label{f55f}
1-    c_{4}(1+  \underline{\lambda} )  \Big(   \underline{\lambda}  ( \epsilon_n+ \delta_n)   +\delta_n|\ln \delta_n| + \epsilon_n|\ln \epsilon_n| \Big)  \ \ \ \ \ \ \ \ \   \ \ \ \ \ \ \ \ \  \ \ \ \  \ \ \ \ \ \ \ \nonumber\\
 \ \ \ \ \ \ \ \ \   \ \ \ \ \ \ \ \ \  \ \ \ \  \ \ \ \ \ \ \  \geq     \exp\bigg\{- c_{5}  (1+  \underline{\lambda} )  \Big(   \underline{\lambda}  ( \epsilon_n+ \delta_n)   +\delta_n|\ln \delta_n| + \epsilon_n|\ln \epsilon_n| \Big) \bigg\} .
\end{eqnarray}
Implementing (\ref{f55f}) in (\ref{f51}), we obtain for all $0\leq \underline{\lambda} \leq  \alpha_1 \min \{\epsilon_n^{-1/2}, \delta_n^{- 1/2} \}$,
\begin{eqnarray}
&& \frac{\mathbf{P}\Big(X_n >x \sqrt{ \langle X\rangle_n } ,    \ \left| \left\langle X\right\rangle _n-1\right|  \leq  \delta_n(\underline{\lambda})  \Big)}{1-\Phi \left( \underline{\lambda}\right)} \nonumber \\
&&\ \ \    \geq \ \exp\bigg\{ - c_{6} \,\bigg( \underline{\lambda}^3 ( \epsilon_n  +\delta_n )+  \underline{\lambda}^2 \delta_n \sqrt{ |\ln \delta_n|}  + (1+  \underline{\lambda} )  \Big(   \underline{\lambda}  ( \epsilon_n+ \delta_n)   +\delta_n|\ln \delta_n| + \epsilon_n|\ln \epsilon_n| \Big) \bigg)\bigg\} \nonumber  \\
&&\ \ \    \geq \ \exp\bigg\{ - c_{7} \,\bigg( \underline{\lambda}^3 ( \epsilon_n  +\delta_n )+   (1+  \underline{\lambda} )  \big(  \delta_n|\ln \delta_n| + \epsilon_n|\ln \epsilon_n| \big) \bigg)\bigg\}.    \label{f54}
\end{eqnarray}

Next, we consider the case   $\alpha_1 \min \{\epsilon_n^{-1/2}, \delta_n^{-1/2} \} \leq  \underline{\lambda} =   o( \min\{\epsilon_n^{-1},   \delta_n^{-1}  \})$. Let $K \geq 1$ be a constant
depending on $\alpha_1$, whose exact value will be chosen later. Clearly, we have
\begin{eqnarray}\label{jknjsta}
&& \mathbf{E}_{\underline{\lambda}}     \left(e^{-\underline{\lambda}Y_n(\underline{\lambda})}\mathbf{1}_{\{ Y_n(\underline{\lambda})>0,  | \left\langle X\right\rangle_n-1 |  \leq \delta_n(\underline{\lambda})   \}} \right)  \geq   \mathbf{E}_{\underline{\lambda}} \Big(e^{-\underline{\lambda}Y_n(\underline{\lambda})}\mathbf{1}_{\{0< Y_n(\underline{\lambda})\leq  K \gamma_n,    | \left\langle X\right\rangle_n-1 |  \leq \delta_n(\underline{\lambda})   \}} \Big) \nonumber\\
&& \ \ \ \ \ \ \ \ \ \ \ \ \ \ \ \ \ \  \ \ \ \ \ \ \ \ \ \ \  \ \ \ \ \ \ \ \ \ \   \geq  \ e^{-\underline{\lambda} K \gamma_n}\mathbf{P}_{\underline{\lambda}} \Big(0< Y_n(\underline{\lambda})\leq  K \gamma_n ,    | \left\langle X\right\rangle_n-1 |  \leq \delta_n(\underline{\lambda})   \Big),
\end{eqnarray}
where $\gamma_n  = \underline{\lambda }( \epsilon_n + \delta_n) +\delta_n|\ln \delta_n|+\epsilon_n|\ln \epsilon_n|.$
Using Lemmas \ref{ghkl} and \ref{LEMMA4}, we deduce that  for all $1\leq \underline{\lambda} =   o( \min\{\epsilon_n^{-1},  \delta_n^{-1}  \}),$
\begin{eqnarray*}
\mathbf{P}_{\underline{\lambda}} \Big(0< Y_n(\underline{\lambda})\leq  K \gamma_n,   | \left\langle X\right\rangle_n-1 |  \leq \delta_n(\underline{\lambda})    \Big) &\geq &  \mathbf{P}  \Big( 0<  \mathcal{N}
\leq  K \gamma_n  \Big)  - c_{1} \gamma_n - \mathbf{P}_{\underline{\lambda}} \Big(   | \left\langle X\right\rangle_n-1 |  > \delta_n(\underline{\lambda})    \Big)\\
  &\geq& \frac{1}{\sqrt{2 \pi}} K \gamma_n  e^{-K^2  \gamma_n ^2/2}   - c_{2} \gamma_n\\
  &\geq& \left( \frac{1}{4} K  - c_{3}\right) \gamma_n.
\end{eqnarray*}
Letting $K\geq   8c_{1} $, it follows that
$$\mathbf{P}_{\underline{\lambda}} \Big(0< Y_n(\underline{\lambda})\leq  K \gamma_n,   | \left\langle X\right\rangle_n-1 |  \leq \delta_n(\underline{\lambda})    \Big) \geq \frac18 K\gamma_n \geq \frac18 K  \frac{ (1+\underline{\lambda })\gamma_n} { 1+\underline{\lambda } }.$$
Let $K=\max \big\{8c_{1},  \frac{8 }{\sqrt{\pi}}\alpha_1^{-2} \big\}$. Taking into account that
$ \alpha_1 \min \{\epsilon_n^{-1/2}, \delta_n^{-1/2} \}\leq  \underline{\lambda} =   o( \min\{\epsilon_n^{-1},  \delta_n^{-1}  \})$, we have $ \frac18 K (1+ \underline{\lambda })\gamma_n  \geq \frac{1}{\sqrt{\pi}}$  and
\begin{eqnarray*}
\mathbf{P}_{\underline{\lambda }} \Big(0< Y_n(\underline{\lambda})\leq  K \gamma_n,     | \left\langle X\right\rangle_n-1 |  \leq \delta_n(\underline{\lambda})  \Big)  \geq    \frac{1}{\sqrt{\pi}(1+\underline{\lambda}) } .
\label{jknjstb}
\end{eqnarray*}
Since the inequality
$\frac{1}{\sqrt{\pi}(1+\lambda) }  e^{- \lambda^2/2}  \geq 1-\Phi \left(  \lambda \right)$ is valid for all $\lambda \geq 0$ (see (\ref{f39})),
it follows that  for all $ \alpha_1 \min \{\epsilon_n^{-1/2}, \delta_n^{-1/2} \} \leq  \underline{\lambda} =     o( \min\{\epsilon_n^{-1},  \delta_n^{-1}  \}),$
\begin{eqnarray}\label{dfac}
\mathbf{P}_{\underline{\lambda }} \Big(0< Y_n(\underline{\lambda})\leq  K \gamma_n,     | \left\langle X\right\rangle_n-1 |  \leq \delta_n(\underline{\lambda})  \Big)  \geq \big( 1-\Phi(  \underline{\lambda}) \big)e^{\underline{\lambda}^2/2} .
\end{eqnarray}
From (\ref{jknjssa}), using the inequalities (\ref{jknjsta}) and (\ref{dfac}), we obtain for all $\alpha_1 \min \{\epsilon_n^{-1/2}, \delta_n^{- 1/2} \} \leq  \underline{\lambda} =     o( \min\{\epsilon_n^{-1},  \delta_n^{-1}  \}),$
 \begin{eqnarray}
 &&\frac{\mathbf{P}\Big(X_n >x \sqrt{ \langle X\rangle_n },    \ \left| \left\langle X\right\rangle _n-1\right| \leq  \delta_n(\underline{\lambda})  \Big)}{1-\Phi \left( \underline{\lambda}\right)}\nonumber \\
  &&\ \ \ \ \  \geq \exp \bigg \{ -c_{8}   \Big( \underline{\lambda}^3 (\epsilon_n+\delta_n)  + \underline{\lambda}^2 \delta_n \sqrt{ |\ln \delta_n|}+ \underline{\lambda}    \big(  \delta_n|\ln \delta_n| + \epsilon_n|\ln \epsilon_n| \big)\Big) \bigg\} . \label{fgj53}
\end{eqnarray}

Putting (\ref{f54}) and (\ref{fgj53}) together, we obtain for all $0 \leq  \underline{\lambda}  =o( \min\{\epsilon_n^{-1},  \delta_n^{-1}  \}),$
\begin{eqnarray}
&&\frac{\mathbf{P}\Big(X_n >x  \sqrt{ \langle X\rangle_n },    \ \left| \left\langle X\right\rangle _n-1\right|  \leq  \delta_n(\underline{\lambda})  \Big)}{1-\Phi \left( \underline{\lambda}\right)} \nonumber \\
&&\ \ \ \ \ \ \   \geq \exp\bigg\{ - c_{9}\, \bigg( \underline{\lambda}^3 ( \epsilon_n  +\delta_n )+  \underline{\lambda}^2 \delta_n \sqrt{ |\ln \delta_n|}  + (1+  \underline{\lambda} )  \big(  \delta_n|\ln \delta_n| + \epsilon_n|\ln \epsilon_n| \big) \bigg) \bigg\} \nonumber\\
&&\ \ \ \ \ \ \    \geq \exp\bigg\{ - c_{10}\, \bigg( \underline{\lambda}^3 ( \epsilon_n  +\delta_n )  + (1+  \underline{\lambda} )  \big(  \delta_n|\ln \delta_n| + \epsilon_n|\ln \epsilon_n| \big) \bigg) \bigg\} ,\label{ft52}
\end{eqnarray}
where the last line follows by the inequality  for all $\underline{\lambda} \geq 0,$
\[
 \underline{\lambda}^2 \delta_n \sqrt{ |\ln \delta_n|} \leq  \underline{\lambda}^3  \delta_n +  \underline{\lambda}  \delta_n|\ln \delta_n|.
 \]
As in the proof of Lemma  \ref{lem41}, we now compare $1-\Phi (\underline{\lambda})$ with  $1-\Phi (x)$. Similar to (\ref{f42}), we have for all $0 \leq  x =    o( \min\{\epsilon_n^{-1},   \delta_n^{-1}  \}),$
\begin{equation}\label{f53}
1-\Phi \left( \underline{\lambda}\right) =\Big( 1-\Phi (x)\Big)  \exp \bigg\{ -  \theta(x)   c_1  (1+x)\Big(
x^2(\epsilon_n + \delta_n)+  x  \delta_n\sqrt{|\ln \delta_n|}  \Big) \bigg\},
\end{equation}
where $0\leq \theta(x) \leq 1.$
From (\ref{ft52}), using (\ref{f45})  and (\ref{f53}), we get for all   $0 \leq  x   =  o( \min\{\epsilon_n^{-1},   \delta_n^{-1}  \}),$
 \begin{eqnarray}
&& \frac{\mathbf{P}\Big(X_n >x \sqrt{ \langle X\rangle_n } ,    \ \left| \left\langle X\right\rangle _n-1\right|  \leq  \delta_n(\underline{\lambda})  \Big)}{1-\Phi \left( x\right)}  \nonumber \\
&&\ \ \    \geq  \exp \bigg \{ - c_{2} \, \bigg( (1+x)x^2  (\epsilon_n  + \delta_n)+ (x+x^2)  \delta_n\sqrt{|\ln \delta_n|} + (1+  x )  \big(  \delta_n|\ln \delta_n| + \epsilon_n|\ln \epsilon_n| \big) \bigg)\bigg\}\nonumber \\
&&\ \ \    \geq  \exp \bigg \{ - c_{3} \, \bigg( x^3  (\epsilon_n  + \delta_n)+   (1+  x )  \big(  \delta_n|\ln \delta_n| + \epsilon_n|\ln \epsilon_n| \big) \bigg)\bigg\}\nonumber.
\end{eqnarray}
Hence, we have for all $0 \leq  x   =  o( \min\{\epsilon_n^{-1},   \delta_n^{-1}  \}),$
 \begin{eqnarray*}
\frac{\mathbf{P}\Big(X_n >x   \sqrt{ \langle X\rangle_n } \Big)}{1-\Phi \left( x\right) } &\geq& \frac{\mathbf{P}\Big(X_n >x \sqrt{ \langle X\rangle_n } ,    \ \left| \left\langle X\right\rangle _n-1\right|  \leq \delta_n(\underline{\lambda}) \Big)}{1-\Phi \left( x\right) } \nonumber \\
&\geq & \exp \bigg \{ - c_{4} \,\bigg( x^3  (\epsilon_n  + \delta_n)+   (1+  x )  \big(  \delta_n|\ln \delta_n| + \epsilon_n|\ln \epsilon_n| \big) \bigg)\bigg\}.
\end{eqnarray*}
This completes the proof  of   Lemma \ref{lem42}.
\end{proof}

 \section{Proof of Corollary \ref{corollary03}}
\setcounter{equation}{0}
Denote $ \gamma_n= (\epsilon_n  + \delta_n)^{1/8}$.
Clearly, it holds
\begin{eqnarray}
&& \sup_{x \in \mathbf{R}}\Big|\mathbf{P}(  X_n/ \sqrt{ \langle X\rangle_n }  \leq  x  ) - \Phi \left(  x\right)\Big| \nonumber \\
& &\leq \sup_{   x > \gamma_n^{-1/8}} \big|\mathbf{P}( X_n/ \sqrt{ \langle X\rangle_n } \leq  x  ) - \Phi \left(  x\right) \big|    + \sup_{  0 \leq x \leq  \gamma_n^{-1/8} } \big|\mathbf{P}( X_n/ \sqrt{ \langle X\rangle_n }  \leq  x  ) - \Phi \left(  x\right)\big| \nonumber\\
&&\ \ \ \ \ \ +  \sup_{ - \gamma_n^{-1/8} \leq x \leq 0 } \big|\mathbf{P}( X_n/ \sqrt{ \langle X\rangle_n }  \leq  x  ) - \Phi \left(  x\right) \big|  +\sup_{   x < - \gamma_n^{-1/8}} \big|\mathbf{P}( X_n/ \sqrt{ \langle X\rangle_n }  \leq  x  ) - \Phi \left(  x\right) \big| \nonumber  \\
&& =: H_1 + H_2+H_3+H_4.   \label{indeq0d10}
\end{eqnarray}
By Theorem \ref{th0}, we deduce that
\begin{eqnarray*}
H_1 &= & \sup_{   x > \gamma_n^{-1/8} } \Big|\mathbf{P} ( X_n/\sqrt{\langle X \rangle_n} > x   )- \big(1 -  \Phi \left( x\right) \big) \Big| \\
  &\leq&   \sup_{   x > \gamma_n^{-1/8} }  \mathbf{P}\big(X_n/\sqrt{\langle X \rangle_n} > x   \big) + \sup_{   x > \gamma_n^{-1/8} }  \big(1 -  \Phi \left( x\right) \big)   \\
  &\leq&   \mathbf{P}\big( X_n/\sqrt{\langle X \rangle_n}  > \gamma_n^{-1/8}   \big) +    \big(1 -  \Phi  ( \gamma_n^{-1/8}  ) \big)\\
  &\leq&   \big(1 -  \Phi  ( \gamma_n^{-1/8}  ) \big)e^c   +    \exp\Big\{ -\frac12  \gamma_n^{-1/4}  \Big \}  \\
 & \leq &  c_1 \big(  \delta_n|\ln \delta_n| + \epsilon_n|\ln \epsilon_n| \big)
\end{eqnarray*}
and
\begin{eqnarray*}
H_4   &\leq & \sup_{   x <- \gamma_n^{-1/8} }  \mathbf{P}\big( X_n/\sqrt{\langle X \rangle_n} \leq x  \big) + \sup_{   x <- \gamma_n^{-1/8} }     \Phi \left( x\right) \\
   &\leq&     \mathbf{P}\big( X_n/\sqrt{\langle X \rangle_n} \leq -\gamma_n^{-1/8}  \big) +     \Phi  ( -\gamma_n^{-1/8}   )  \\
 &\leq &   \Phi (- \gamma_n^{-1/8} )  e^c   +    \exp\Big\{ -\frac12 \gamma_n^{-1/4} \Big\} \\
 & \leq &  c_2  \big(  \delta_n|\ln \delta_n| + \epsilon_n|\ln \epsilon_n| \big).
\end{eqnarray*}
By Theorem \ref{th0} and the inequality $|e^x-1|\leq |x|e^{|x|},$ we get
\begin{eqnarray*}
H_2 &=&\sup_{0\leq   x \leq \gamma_n^{-1/8} } \Big|\mathbf{P}\big( X_n/\sqrt{\langle X \rangle_n} > x  \big)- \big(1 -  \Phi \left( x\right) \big) \Big| \nonumber \\
 &\leq&  \sup_{0\leq   x \leq \gamma_n^{-1/8} } c_{ 1} \Big(1-\Phi(x) \Big) \Big(    x^3 (\epsilon_n  + \delta_n)  + (1+ x)  \big(  \delta_n|\ln \delta_n| + \epsilon_n|\ln \epsilon_n| \big)    \Big)    \nonumber\\
 &\leq & c_{3}   \, \big(  \delta_n|\ln \delta_n| + \epsilon_n|\ln \epsilon_n| \big)
 \end{eqnarray*}
 and
 \begin{eqnarray*}
H_3 &=& \sup_{ - \gamma_n^{-1/8} \leq x \leq 0 } \big|\mathbf{P}\big( X_n/\sqrt{\langle X \rangle_n} \leq x  \big) -  \Phi \left( x\right) \big|  \nonumber \\
   &\leq&   \sup_{ - \gamma_n^{-1/8} \leq x \leq 0 } c_{ 1}  \Phi(x) \Big(    x^3  (\epsilon_n  + \delta_n)   + (1+ x) \big(  \delta_n|\ln \delta_n| + \epsilon_n|\ln \epsilon_n| \big)   \Big)   \nonumber\\
 &\leq & c_{ 4} \, \big(  \delta_n|\ln \delta_n| + \epsilon_n|\ln \epsilon_n| \big).
 \end{eqnarray*}
Applying  the upper bounds  of $H_1, H_2, H_3$ and $H_4$ to (\ref{indeq0d10}),   we obtain the desired  inequality. This completes the proof of Corollary \ref{corollary03}.

\section{Proofs of Theorems \ref{thmg2} and \ref{thmg345}} \label{secp}
\setcounter{equation}{0}

\subsection{Some lemmas}
By the well-known Stirling's formula
\[
\ln  \Gamma(x)  =(x-\frac12) \ln x -x + \frac12 \ln 2 \pi + O(\frac{1}{x})  \ \ \ \  \ \textrm{as}\ \ x \rightarrow \infty,
\]
 we deduce that for $  p\in(0, 1]$,
\begin{eqnarray}\label{a10}
\lim\limits_{n\to \infty} a_n n^{ 2p-1} =\Gamma(2p).
\end{eqnarray}
Moreover, for $  p\in(0, 3/4)$, we have
\begin{eqnarray}\label{a15}
\lim\limits_{n\to \infty} \frac{ v_n}{ n^{3-4p}} = \frac{\Gamma{(2p)}^2}{3-4p},
\end{eqnarray}
and, for $ p=3/4$,  it holds
\begin{eqnarray}\label{a16}
\lim\limits_{n\to \infty} \frac{ v_n}{  \ln n}    =\frac{\pi}{4}.
\end{eqnarray}
See also Bercu \cite{B18e} for the equalities (\ref{a10})-(\ref{a16}).
Denote
\[
\gamma _n=1+\frac{2p-1}{n},\ \ \ \ \ \ \ \ \ n\geq 1.
\]
It is easy to see that  for all $n\geq 2,$
\[
a_n= \prod_{i=1}^{n-1}\frac{1}{\gamma_i}.
\]
Define the filtration $\mathcal{F}_0=\sigma\{\emptyset, \Omega\}, \mathcal{F}_k=\sigma\{  \alpha_{i }, \beta_{i }, Z_{i} : 1\leq i\leq k\}, 1\leq k \leq n,  $ and
\begin{equation}\label{mndef}
M_n=  a_nS_n.
\end{equation}
  For $ 0\leq k \leq n,$  set
\begin{eqnarray}
M_{n,k} & =& \mathbf{E}[ M_{n}|\mathcal{F}_k ]  \nonumber\\
&=&a_n \sum_{i=1}^k \alpha_{i  }X_{\beta_{i  }}(Z_{i}-1) +  a_kT_k  . \nonumber
\end{eqnarray}
  Then  $M_n=M_{n,n}$ and  $(M_{n,k}, \mathcal{F}_k)_{1\leq k \leq n}$ is a martingale.
Moreover, the martingale $(M_{n,k})_{1\leq k \leq n}$  can be rewritten  in the following additive form
\begin{equation}\label{fsdfsdf}
M_{n,k}=\sum_{i=1}^{k}\Big( a_n\alpha_{i  }X_{\beta_{i  }}(Z_{i}-1)+  a_i \varepsilon_i \Big),
\end{equation}
where $\varepsilon_i=T_i-\gamma_{i-1}T_{i-1}$ with convention $\gamma_0T_0=0.$
Set $(\Delta M_{n,k})_{ 0 \leq k \leq n}$ be the martingale differences defined by $\Delta M_{n,1}=M_{n,1}$ and  for all $  2\leq k \leq n,$
\[
\Delta M_{n,k}=M_{n,k}-M_{n,k-1}=a_n\alpha_{k  }X_{\beta_{k }}(Z_{k}-1)+  a_k\varepsilon_k.
\]
In the proof  of Theorem  \ref{thmg2}, we need the following lemma  for the boundness of martingale differences.
\begin{lemma}\label{lem1}
Assume that $0\leq Z_1\leq C$ for some constant $C\geq 1$. For all $1\leq k \leq n $ and $p \in (0, 1]$, it holds
\[
\|\Delta M_{n,k} \|_\infty \leq   C (a_n+2a_k).
\]
In particular, it implies that $\|a_k\varepsilon_k  \|_\infty \leq   2Ca_k  $.
\end{lemma}
 \begin{proof}
It holds obviously that for $n=1$,
\[
\|\Delta M_{n,1} \|_\infty =\| a_n\alpha_{1  }X_{\beta_{1 }}(Z_{1}-1)+  a_1\varepsilon_1  \|_\infty \leq C (a_n+a_1).
\]
Observe that for all $  2 \leq k \leq n,$
\begin{eqnarray}
	\Delta M_{n,k} &=& a_n\alpha_{k  }X_{\beta_{k }}(Z_{k}-1)+  a_k\varepsilon_k\nonumber \\
&=&a_n\alpha_{k  }X_{\beta_{k }}(Z_{k}-1)+  a_kT_k-a_{k-1}T_{k-1}\nonumber \\
	 &=& a_n\alpha_{k  }X_{\beta_{k }}(Z_{k}-1)+  a_k X_k-a_k\frac{T_{k-1}}{k-1} (2p-1). \label{dsd5}
\end{eqnarray}
Notice that   $\|T_{n }\|_\infty=\sum_{i=0}^{n }\|X_{i}\|_\infty \leq  n   $. Thus for $  2 \leq k \leq n,$ it holds $\|\Delta M_{n,k} \|_\infty\leq C (a_n+2a_k).$ From the proof (\ref{dsd5}), we find that $\|a_k\varepsilon_k  \|_\infty \leq   2Ca_k $.
This completes the proof of Lemma \ref{lem1}.
\end{proof}

Denote by $\langle M\rangle_n $   the quadratic variation of $(M_{n,k}, \mathcal{F}_{k})_{0\leq k \leq n}$, that is
\[
\langle M\rangle_n=\sum_{i=1}^n\mathbf{E}[\Delta M_{n,i}^2|\mathcal{F}_{i-1}].
\]
\begin{lemma}\label{lem2} Assume that $\mathbf{E}  Z_1^2 < \infty $. Then, it holds for all $x>0,$
\begin{eqnarray*}
&&   \mathbf{P}\bigg( \big| \langle M\rangle_n  -(v_n  + na_n^2 \sigma^2 )    \big|\geq x (v_n  + na_n^2 \sigma^2 )  \bigg)\\
&&\ \ \ \ \ \ \ \ \ \ \ \ \  \ \ \ \ \ \ \ \ \  \leq\   \left\{ \begin{array}{ll}
\displaystyle c_1 \exp\bigg\{-   c_2 x \frac{n a_n^2\sigma^2 +v_n}{  a_n^2}   \bigg\} \ \ \ \   & \textrm{$0< p < 1/2$}\\
0\ \ \ \   & \textrm{$  p = 1/2$}\\
\displaystyle  c_1 \exp\bigg\{-c_2 (3-4p)(n a_n^2\sigma^2 +v_n) x     \bigg\}  \ \ \ \    & \textrm{$1/2 < p < 3/4$}\\
\displaystyle c_1\,   \exp\bigg\{- c_2 (n a_n^2\sigma^2 +v_n) x   \bigg\} \ \ \ \    & \textrm{$  p =3/4$} .
\end{array} \right.
\end{eqnarray*}

\end{lemma}
\begin{proof}
From (\ref{fsdfsdf}), we get $$\Delta M_{n,k} = a_n\alpha_{k  }X_{\beta_{k  }}(Z_{k}-1) + a_k\varepsilon_k=a_n\alpha_{k}X_{\beta_{k  }}(Z_{k}-1) + a_k(T_k-\gamma_{k-1}T_{k-1}). $$ Thus, it holds for $1\leq k \leq n,$
\begin{eqnarray*}
	\mathbf{E}[ \Delta M_{n,k} ^2   |\mathcal{F}_{k-1}]&=&a_n^2  \mathbf{E}[ (Z_{k}-1)^2]+  a_k^2 \mathbf{E}[ (S_k-\gamma_{k-1}S_{k-1})^2 |  \mathcal{F}_{k-1}]\\
&=&a_n^2\sigma^2+ a_k^2 \big( \mathbf{E}[T_k^2|  \mathcal{F}_{k-1}]  -2\gamma_{k-1}T_{k-1}\mathbf{E}[ T_k |  \mathcal{F}_{k-1}]+\gamma_{k-1}^2T_{k-1}^2 \big) .
\end{eqnarray*}
It is easy to see that
\begin{eqnarray*}
  \mathbf{E}[T_k^2|  \mathcal{F}_{k-1}] &=& \mathbf{E}[ (T_{k-1}+\alpha_kX_{\beta_k} )^2|\mathcal{F}_{k-1}] \\
  &=& T_{k-1}^2+ 2 T_{k-1} \mathbf{E}[  \alpha_k X_{\beta_k}   |\mathcal{F}_{k-1}] + 1  \\
  &=&T_{k-1}^2+ 2 \frac{(2p-1) }{k-1}  T_{k-1}^2   + 1 \\
  &=&  (2 \gamma_{k-1} -1) T_{k-1}^2    +  1
\end{eqnarray*}
and
\begin{eqnarray*}
\mathbf{E}[ T_k |  \mathcal{F}_{k-1}] &=& \mathbf{E}[  T_{k-1}+\alpha_kX_{\beta_k}  |\mathcal{F}_{k-1}]
   =  T_{k-1} +   \frac{2p-1 }{k-1}  T_{k-1} =  \gamma_{k-1}   T_{k-1}.
\end{eqnarray*}
Thus, we have $\mathbf{E}[ \Delta M_{n,1} ^2   ] =  1 +a_n^2 \sigma^2 $ and for $k\geq2,$
\begin{eqnarray}
\mathbf{E}[ \Delta M_{n,k} ^2   |\mathcal{F}_{k-1}]
&=&a_n^2\sigma^2+ a_k^2 \big( (2 \gamma_{k-1} -1) T_{k-1}^2  +  1   -2\gamma_{k-1}^2T_{k-1}^2 +\gamma_{k-1}^2T_{k-1}^2 \big) \nonumber  \\
&=&a_n^2\sigma^2+ a_k^2 \big( 1-( \gamma_{k-1} -1)^2 T_{k-1}^2 \big) \nonumber \\
&=& a_n^2\sigma^2+ a_k^2  - ( 2p -1)^2 a_k^2 ( \frac{T_{k-1}}{k-1} )^2   .  \label{ddfsds}
\end{eqnarray}
 Since $\frac{a_{n+1}}{a_n}\sim 1$ as $n\rightarrow \infty$ (cf.\ (\ref{a10})),  by the definitions of $v_n$, we obtain
\begin{eqnarray*}
	\langle M\rangle_n&=&n a_n^2\sigma^2 +v_n-(2p-1)^2 \bigg(\sum_{k=1}^{n-1}\Big (\frac{a_{k+1}}{a_k}\Big)^2\Big(\frac{a_kT_k}{k}\Big)^2\bigg)\\
 &=& n a_n^2\sigma^2 +v_n-O(1)(2p-1)^2  \sum_{k=1}^{n-1} \Big(\frac{a_kT_k}{k}\Big)^2 .
\end{eqnarray*}
From the last line, we have for all $t \geq 1,$
\begin{equation}\label{hdfbc}
\Big\|\langle M\rangle_n- (n a_n^2\sigma^2 +v_n ) \Big\|_t \leq c_1 \Big\|\sum_{k=1}^{n-1}(\frac{a_kT_k}{k})^2\Big\|_t \leq c_2\sum_{k=1}^{n-1}\frac{1}{k^2}\|a_kT_k\|_{2t}^2.
\end{equation}
 Using  Rio's inequality (cf. Theorem 2.1 of \cite{Rio09}),   we derive that for all  $t \geq  1,$
\begin{eqnarray*}
	\|a_kT_k \|_{2t}^2\leq  (2t-1) \sum_{i=1}^k \| a_i\varepsilon_i\|_{2t}^2  .
\end{eqnarray*}
By the fact $\| a_i\varepsilon_i  \|_\infty \leq 2 C  a_i $ (cf.\ Lemma \ref{lem1}),   we deduce that  for all  $t \geq  1,$
\begin{eqnarray*}
	\|a_kT_k \|_{2t}^2\leq (2t-1)4   C ^2  v_k.
\end{eqnarray*}
By (\ref{hdfbc}) and (\ref{a15}), it is easy to see  that for $0< p<3/4 $,  it holds  for all $t\geq 1,$
\begin{eqnarray}\label{gdfbcx}
\Big\|\langle M\rangle_n-(n a_n^2\sigma^2 +v_n ) \Big \|_t &\leq& 4 c_{2} C^2(2t-1)\sum_{k=1}^{n-1}\frac{1}{k^2}v_k \nonumber  \\
&\leq&4 c_{3} C^2(2t-1 ) \frac{\Gamma{(2p)}^2}{3-4p}\sum_{k=1}^{n-1}k^{1-4p} \nonumber \\
&\leq& \left\{ \begin{array}{ll}
\displaystyle  c_{4} (2t-1)   n^{2(1-2p)} \ \ \ \   & \textrm{$0< p < 1/2$}\\
\\
\displaystyle  \frac{c_{4} }{3-4p}(2t-1) \ \ \ \    & \textrm{$1/2< p < 3/4 $}.
\end{array} \right.
\end{eqnarray}
Similarly, by (\ref{hdfbc}) and (\ref{a16}), when $p=3/4,$ we have
\begin{eqnarray}\label{dsdffbcx}
\Big\|\langle M\rangle_n-(n a_n^2\sigma^2 +v_n )  \Big \|_t \ \leq\  4 c_{2} C^2(2t-1)\sum_{k=1}^{n-1}\frac{1}{k^2} \ln k \   \leq \    c_{4} (2t-1)  .\nonumber
\end{eqnarray}
For $0< p<3/4 $ and all $\lambda, x>0$,
\begin{eqnarray*}
\mathbf{P}\Big(|\langle M\rangle_n -(n a_n^2\sigma^2 +v_n )      |\geq x (n a_n^2\sigma^2 +v_n )   \Big)	 &\leq& e^{-\lambda x (n a_n^2\sigma^2 +v_n )  }  \mathbf{E} e^{ \lambda |\langle M\rangle_n  -(n a_n^2\sigma^2 +v_n )     | }\\
&  =& e^{-\lambda x (n a_n^2\sigma^2 +v_n )  } \sum_{t=0}^{\infty}\frac{ \lambda^t }{t !} \|\langle M\rangle_n-(n a_n^2\sigma^2 +v_n )   \|_t^t .
\end{eqnarray*}
For $0< p < 1/2,$ by the fact $n!\sim  \sqrt{2\pi n } n^n e^{-n} $ and inequality (\ref{gdfbcx}),  we have for $\lambda=(4 e a_n^2 c_5)^{-1}$ and   all $x>0,$
\begin{eqnarray*}
\mathbf{P}\Big(|\langle M\rangle_n -(n a_n^2\sigma^2 +v_n )      |\geq x (n a_n^2\sigma^2 +v_n )   \Big)	 &\leq&c_1 e^{-\lambda x (n a_n^2\sigma^2 +v_n )   } \sum_{t=0}^{\infty}\frac{ \lambda^t }{t !} (2t-1)^t    n^{2t(1-2p)} c_4^t\\
&\leq&c_2 e^{-\lambda x (n a_n^2\sigma^2 +v_n )   } \sum_{t=0}^{\infty}\frac{ 2^t\lambda^t }{t !}   t^t    (a_n)^{2t } c_5^t \\
&\leq& c_2 e^{-\lambda x (n a_n^2\sigma^2 +v_n )  } \sum_{t=0}^{\infty}2^t\lambda^t  e^t    (a_n)^{2t } c_5^t\\
&\leq& c_3 \exp\bigg\{-   \frac{(n a_n^2\sigma^2 +v_n )  x}{4ec_5 a_n^2}  \bigg\}.
\end{eqnarray*}
When $p=1/2,$   by (\ref{ddfsds}),  we have
\begin{eqnarray*}
 \langle M\rangle_n =n a_n^2\sigma^2 +v_n  ,
\end{eqnarray*}
which gives the desired inequality with $p=1/2.$
  Similarly, for $   1/2 < p < 3/4$, we have for $\lambda=(3-4p)(4 e c_4 )^{-1}$ and  all $x>0,$
\begin{eqnarray*}
\mathbf{P}\Big(|\langle M\rangle_n  -(n a_n^2\sigma^2 +v_n)      |\geq x(n a_n^2\sigma^2 +v_n)   \Big)	
&\leq&c_1\, e^{-\lambda x (n a_n^2\sigma^2 +v_n)  } \sum_{t=0}^{\infty}\frac{ \lambda^t }{t !} (2t-1)^t   c_4^{ t }(3-4p)^{-t}\\
&\leq& c_2 \exp\bigg\{-(3-4p)\frac{ (n a_n^2\sigma^2 +v_n)  x }{4 e c_4}   \bigg\} .
\end{eqnarray*}
 For $p=3/4,$ it holds  for   $\lambda=(4 e c_4)^{-1}$ and all $x>0,$
\begin{eqnarray*}
\mathbf{P}\Big(|\langle M\rangle_n  -(n a_n^2\sigma^2 +v_n)     |\geq x (n a_n^2\sigma^2 +v_n)    \Big)&\leq&c_1\, e^{-\lambda x v_n  } \sum_{t=0}^{\infty}\frac{ \lambda^t }{t !} (2t-1)^t   c_4^{ t }\\	
&\leq&c_2 \exp\bigg\{- \frac{ (n a_n^2\sigma^2 +v_n)  x }{4 e c_4}   \bigg\} .
\end{eqnarray*}
This completes the proof of Lemma \ref{lem2}.
\end{proof}

The following lemma is used in the proof of self-normalized type Cram\'{e}r's moderate deviations.
\begin{lemma}\label{ledyikl}
Assume that $0\leq Z_1\leq C$ for some constant $C\geq 1$. Then for all $x>0,$ it holds
\begin{eqnarray*}
\mathbf{P}\bigg( \Big|  v_n + a_n^2 \sum_{i=1}^n (Z_i-1)^2  - \big(v_n + na_n^2 \sigma^2 \big)    \Big|\geq x     \bigg)
 \leq  2 \exp\bigg\{- \frac{ 2 \, x^2}{ na_n^2 C^4 }   \bigg\} .
\end{eqnarray*}
\end{lemma}
\begin{proof}
 It is easy to see that
\[
 v_n + a_n^2 \sum_{i=1}^n (Z_i-1)^2  - \big(v_n + na_n^2 \sigma^2 \big)    =  a_n^2 \sum_{i=1}^n\big((Z_i-1)^2 -  \sigma^2  \big).
\]
Notice that
 \[
  -  \sigma^2    \leq (Z_i-1)^2 -  \sigma^2  \leq  C^2 -  \sigma^2 .
  \]
The desired inequality is a simple consequence of   Hoeffding' inequality (cf.\, Theorem 2 of \cite{Ho63}).
This completes the proof of Lemma \ref{ledyikl}.
\end{proof}

\subsection{Proof of Theorem \ref{thmg2}}
 Clearly, it holds
\[
\frac{a_n S_n}{\sqrt{v_n + na_n^2 \sigma^2}\ }  =  \frac{M_n  }{\sqrt{v_n + na_n^2 \sigma^2}\ } = \sum_{i=1}^n  \xi_i  ,
\]
 where $\xi_i= \frac{ \Delta M_{n,i} }{\sqrt{v_n + na_n^2 \sigma^2}\ }, i=1,...,n.$ Then $(\xi_i , \mathcal{F}_{i})_{i=1,...,n}$ is  a finite sequence of martingale differences.
By Lemma \ref{lem1}, we have
\[
\|\xi_i \|_\infty \leq  C \max_{1\leq i \leq n} \frac{2a_k+a_n}{\sqrt{v_n + na_n^2 \sigma^2}}  =: \epsilon_n .
\]
Using the inequalities (\ref{a10})-(\ref{a16}), we deduce that
\begin{eqnarray}\label{ineq37}
\displaystyle
\epsilon_n   \asymp  \left\{ \begin{array}{ll}\displaystyle
  \ n^{-1/2} \ \   & \textrm{if $0<p \leq1/2  $ }\\
\sqrt{ 3-4p  } \ n^{- (3-4p)/2} \ \   & \textrm{if $1/2 < p <3/4  $ } \\
\displaystyle    ( \ln n)^{-1/2}  \ \   & \textrm{if $  p=3/4$.}
\end{array} \right.
\end{eqnarray}
Moreover, from Lemma \ref{lem2},  we have for all $x>0,$
\begin{eqnarray*}
  \mathbf{P}\bigg( \Big|\sum_{i=1}^n\mathbf{E}( \xi_i^2| \mathcal{F}_{i-1}) -1   \Big |\geq x   \bigg)
&\leq& \left\{ \begin{array}{ll}
\displaystyle c_1 \exp\bigg\{-   c_2 x \frac{v_n + na_n^2 \sigma^2}{  a_n^2}   \bigg\} \ \ \ \   & \textrm{$0< p < 1/2$}\\
0\ \ \ \   & \textrm{$  p = 1/2$}\\
\displaystyle  c_1 \exp\bigg\{-c_2 (3-4p)( v_n + na_n^2 \sigma^2)x     \bigg\}  \ \ \ \    & \textrm{$1/2 < p < 3/4$}\\
\displaystyle c_1\,   \exp\bigg\{- c_2 (v_n + na_n^2 \sigma^2) x   \bigg\} \ \ \ \    & \textrm{$  p =3/4$} .
\end{array} \right.
\end{eqnarray*}
Using the inequalities (\ref{a10})-(\ref{a16}), we deduce that for all $x>0,$
\begin{eqnarray*}
  \mathbf{P}\bigg( \Big|\sum_{i=1}^n\mathbf{E}( \xi_i^2| \mathcal{F}_{i-1}) -1    \Big|\geq x   \bigg)
&\leq&  c_1\,   \exp\bigg\{- c_2\, \delta_n^{-2} x   \bigg\},
\end{eqnarray*}
where $\delta_n$ satisfies
\begin{eqnarray}
\displaystyle
\delta_n   \asymp  \left\{ \begin{array}{ll}\displaystyle
  \ n^{-1/2} \ \   & \textrm{if $0<p \leq1/2  $ }\\
  (3-4p )^{ -1/2}    \ n^{- (3-4p)/2} \ \   & \textrm{if $1/2 < p <3/4  $ } \\
\displaystyle    ( \ln n)^{-1/2}  \ \   & \textrm{if $  p=3/4$.}
\end{array} \right.
\end{eqnarray}
Applying Theorem \ref{th2s} and Remark \ref{fdsdds} to $\frac{a_n S_n}{\sqrt{v_n + na_n^2 \sigma^2}}$,
 we obtain the desired inequalities. This completes the proof of Theorem \ref{thmg2}.

\subsection{Proof of Theorem \ref{thmg345}}
We first give a proof for the case $0< p \leq  1/2.$
Assume that $ \varepsilon_x \in (0, 1/2].$
It is easy to see that for all $x\geq 0,$
\begin{eqnarray}
 && \mathbf{P}\bigg( \frac{a_nS_n}{ \sqrt{v_n + a_n^2 \Pi_{i=1}^n (Z_i-1)^2}} \geq\ x \bigg) \nonumber \\
  && =  \mathbf{P}\bigg( \frac{a_nS_n}{ \sqrt{v_n + a_n^2 \Pi_{i=1}^n (Z_i-1)^2}} \geq\ x ,\ v_n + a_n^2 \Pi_{i=1}^n (Z_i-1)^2 \geq (1- \varepsilon_x)(v_n + na_n^2 \sigma^2 ) \bigg) \nonumber  \\
 && + \ \mathbf{P}\bigg( \frac{a_nS_n}{ \sqrt{v_n + a_n^2 \Pi_{i=1}^n (Z_i-1)^2}} \geq\ x ,\  v_n + a_n^2 \Pi_{i=1}^n (Z_i-1)^2<(1- \varepsilon_x)(v_n + na_n^2 \sigma^2)  \bigg) \nonumber  \\
 &&\leq\mathbf{P}\bigg(\frac{a_nS_n}{ \sqrt{v_n + n a_n^2 \sigma^2}} \geq x \sqrt{ 1- \varepsilon_x  \ }   \bigg) \nonumber  \\
 && + \ \mathbf{P}\bigg( v_n + a_n^2 \Pi_{i=1}^n (Z_i-1)^2<(1- \varepsilon_x)(v_n + na_n^2 \sigma^2) \bigg) \nonumber  \\
 &&=: P_1 +P_2.
\end{eqnarray}
By Theorem  \ref{thmg2} and an argument similar to (\ref{f41}), we have for all $x\geq 0,$
\begin{eqnarray}\label{fsds01}
 P_1   & \leq& \Big(1-\Phi(x\sqrt{1- \varepsilon_x })\Big)\exp\bigg\{    c  \bigg( \frac{x^3}{\sqrt{n}}   +  (1+x) \frac{ \ln n }{\sqrt{n}}  \bigg)  \bigg\} \nonumber  \\
 & \leq& \Big(1-\Phi(x)\Big)\exp\bigg\{    c  \bigg(x \varepsilon_x   +  \frac{x^3}{\sqrt{n}}   +  (1+x) \frac{ \ln n }{\sqrt{n}}  \bigg)  \bigg\} .
\end{eqnarray}
Using Lemma \ref{ledyikl}, we get for all $x\geq 0,$
\begin{eqnarray}\label{fsds02}
 P_2   & \leq& 2 \exp\bigg\{- \frac{ 2 (v_n   + na_n^2\sigma^2)^2}{ na_n^2C^4 } \varepsilon_x^2  \bigg\}.
\end{eqnarray}
Taking $\varepsilon_x= c_0  (x + \sqrt{\ln n } )/ \sqrt{n}\, $ with $c_0$ large enough, by (\ref{a15}), (\ref{fsds01}) and (\ref{fsds02}),
we deduce that for all $0\leq x =o( \sqrt{n}),$
\begin{eqnarray}
 \mathbf{P}\Bigg(\frac{a_nS_n}{ \sqrt{v_n + a_n^2 \Pi_{i=1}^n (Z_i-1)^2}} \geq x \ \Bigg)
 &\leq&\Big(1-\Phi(x)\Big)\exp\bigg\{    c_1  \bigg(   \frac{x^3}{\sqrt{n}}   +  (1+x) \frac{ \ln n }{\sqrt{n}}  \bigg)  \bigg\} \nonumber  \\
 && + \, 2 \, \exp\bigg\{- c_2 c_0^2 (x+ \sqrt{\ln n} )  ^2 n^{2-4p} \bigg\}. \nonumber
\end{eqnarray}
Applying (\ref{f39}) to the last inequality, we obtain for all $0\leq x =o( \sqrt{n}),$
\begin{eqnarray}\label{dfschj}
 \mathbf{P}\Bigg(\frac{a_nS_n}{ \sqrt{v_n + a_n^2 \Pi_{i=1}^n (Z_i-1)^2}} \geq x  \Bigg)
  \leq \Big(1-\Phi(x)\Big)\exp\bigg\{    c_3  \bigg(   \frac{x^3}{\sqrt{n}}   +  (1+x) \frac{ \ln n }{\sqrt{n}}  \bigg)  \bigg\} .
\end{eqnarray}
Next, we consider the case $1/2< p < 3/4.$
Taking $\varepsilon_x= c_0  (x + \sqrt{\ln n } )/ \sqrt{v_n}\, $ with $c_0$ large enough, by an argument similar to
the proof of (\ref{dfschj}), we get for all $0\leq x =o( \sqrt{v_n}),$
\begin{eqnarray}
 \mathbf{P}\Bigg(\frac{a_nS_n}{ \sqrt{v_n + a_n^2 \Pi_{i=1}^n (Z_i-1)^2}} \geq x \Bigg)
  \leq \Big(1-\Phi(x)\Big)\exp\bigg\{    c_p  \bigg(   \frac{x^3}{\sqrt{v_n}}   +  (1+x) \frac{ \ln n }{\sqrt{v_n}}  \bigg)  \bigg\} . \nonumber
\end{eqnarray}
For $  p = 3/4,$  the proof is similar, but $\varepsilon_x= c_0  (x + \sqrt{\ln \ln n } )/  \sqrt{\ln n }.$
Then we obtain the desired upper bounds for the tail probability $\mathbf{P}\Big(\frac{a_nS_n}{ \sqrt{v_n + a_n^2 \Pi_{i=1}^n (Z_i-1)^2}} \geq x \Big), x\geq 0.$
Notice that for all $x\geq 0,$
\begin{eqnarray}
 && \mathbf{P}\Bigg(\frac{a_nS_n}{ \sqrt{v_n + a_n^2 \Pi_{i=1}^n (Z_i-1)^2}} \geq x \Bigg)  \\
  & \geq&   \ \mathbf{P}\bigg(\frac{a_nS_n}{ \sqrt{v_n + a_n^2 \Pi_{i=1}^n (Z_i-1)^2}} \geq x,\  v_n + a_n^2 \Pi_{i=1}^n (Z_i-1)^2<(1+ \varepsilon_x)(v_n + na_n^2 \sigma^2)   \bigg) \nonumber  \\
 &\geq&\mathbf{P}\bigg( \frac{a_nS_n }{\sqrt{v_n + n a_n^2\sigma^2}} \geq x \sqrt{ 1+ \varepsilon_x  \ } ,\  v_n + a_n^2 \Pi_{i=1}^n (Z_i-1)^2<(1+ \varepsilon_x)(v_n + na_n^2 \sigma^2)   \bigg) \nonumber  \\
 &\geq&\mathbf{P}\bigg(  \frac{a_nS_n }{\sqrt{v_n + n a_n^2\sigma^2}} \geq x \sqrt{ 1+ \varepsilon_x  \ }   \bigg)
   - \ \mathbf{P}\bigg(  v_n + a_n^2 \Pi_{i=1}^n (Z_i-1)^2 \geq (1+ \varepsilon_x)(v_n + na_n^2 \sigma^2)    \bigg).\nonumber
\end{eqnarray}
Thus the  desired  lower bounds for the tail probability $\mathbf{P}\big(a_nS_n \geq x \sqrt{v_n + a_n^2 \Pi_{i=1}^n (Z_i-1)^2}\, \big), x\geq0, $ can be
obtained by a similar argument.
The proof for  $\mathbf{P}\big(-a_nS_n \geq x \sqrt{v_n + a_n^2 \Pi_{i=1}^n (Z_i-1)^2}\, \big), x\geq0, $   follows by a similar argument. This completes the proof of Theorem \ref{thmg345}.

\section{\textbf{Proofs of Theorem \ref{ththeta2} and its corollaries}\label{sec4.2}}
\setcounter{equation}{0}

\subsection{Proof of Theorem \ref{ththeta2}}
By (\ref{eqAuto}), we have $X_k= \sum_{i=0}^k \theta^{k-i } \varepsilon_i. $ Taking into account that $|\theta| < 1$ and $|\varepsilon_n | \leq H, $  we deduce that for all $k\geq 1,$
\[
|X_k| \leq H\sum_{i=0}^k |\theta|^{k-i } \leq \frac{H} {1-|\theta|}.
\]
From (\ref{eqAuto}), it is easy to see that
\[
\sum_{k=1} ^n X_{k}X_{k-1}=\sum_{k=1}^n ( \theta X^2_{k-1}+    X_{k-1}\varepsilon_{k}),
\]
from which we deduce that for all $n \geq 1$,
\begin{equation}\label{eqD1}
\hat{\theta}_n-\theta= \frac{ \sum_{k=1} ^n X_{k-1} \varepsilon_k }{ \sum_{k=1}^n X_{k-1}^2 }.
\end{equation}
For any $i \geq 1$, denote by
\[
\eta_i= X_{i-1} \varepsilon_i, \  \,\,\,\,  S_n = \sum_{i=1}^n \eta_i  \ \ \ \ \textrm{and}\,\,\,\ \ \  \mathcal{F}_{i}=\sigma \big(  \varepsilon_k, 0\leq k\leq i \big).
\]
Then $(\eta_i,\mathcal{F}_{i})_{i\geq 1}$ is a sequence of martingale differences  and  satisfies
\[
|\eta_i| \leq \frac{H^2} {1-|\theta|}\ \ \ \ \ \
\textrm{and} \ \ \ \ \
 \langle S\rangle_n = \sum_{i=1}^n \mathbf{E} \big( \eta_i^2  \big|  \mathcal{F}_{i-1} \big)=  \sigma^2   \sum_{i=1}^n X_{i-1}^2.
\]
By some simple calculations, we get
\[
\mathbf{E} S_n^2=  \sigma^2   \sum_{i=1}^n \mathbf{E} X_{i-1}^2= \sigma^4   \sum_{i=1}^n  \sum_{j=0}^i   \theta^{2(i-j) }  = \frac{n \sigma ^4 }{1- \theta^2}\Big(1+ \frac{\theta^{2(n+2)}}{n(1-\theta^2)} \Big)   \sim \frac{n \sigma ^4 }{1- \theta^2},\ \ \ n \rightarrow \infty .
\]
Using   Doob's decomposition theorem, we have
\begin{eqnarray*}
\sum_{i=1}^{n} X_{i-1}^2 - \mathbf{E}\Big(\sum_{i=1}^{n} X_{i-1}^2 \Big)&=& \sum_{k=1}^{n} M_k ,
\end{eqnarray*}
where
\[
 M_k =  \mathbf{E}\Big(\sum_{i=1}^{n} X_{i-1}^2 \Big| \mathcal{F}_k \Big) - \mathbf{E}\Big(\sum_{i=1}^{n} X_{i-1}^2 \Big| \mathcal{F}_{k-1} \Big)  . \]
Notice that $(M_k, \mathcal{F}_k)_{ 0 \leq k \leq n}$ is a finite sequence of martingale differences.
It is easy to see that
\begin{eqnarray*}
M_k&=&  \mathbf{E}\Big(\sum_{i=k+1}^{n} X_{i-1}^2 \Big| \mathcal{F}_k \Big) - \mathbf{E}\Big(\sum_{i=k+1}^{n} X_{i-1}^2 \Big| \mathcal{F}_{k-1} \Big)    \\
&=&  \sum_{i=k +1 }^{n} \bigg( 2\varepsilon_k\varepsilon_{k-1}\theta^{2i - 2k-1} +  ( \varepsilon_k^2-\mathbf{E}\varepsilon_k^2 ) \theta^{2i - 2k-2}  \bigg  )  \\
&=&  \sum_{i=k +1 }^{n} \bigg( 2\varepsilon_k\varepsilon_{k-1}\theta  +  ( \varepsilon_k^2-\mathbf{E}\varepsilon_k^2 ) \bigg  )\theta^{2i - 2k-2}
\end{eqnarray*}
and that
\begin{eqnarray*}
|M_k|&\leq &    \sum_{i=k +1 }^{n} \big(2M^2 |\theta| +M^2 \big  )\theta^{2i - 2k-2} \leq  \frac{2 |\theta| +1}{1-\theta^2} M^2  .
\end{eqnarray*}
By Azuma-Hoeffding's inequality, we have for all $x>0,$
\[
 \mathbf{P}\bigg( \Big| \sum_{i=1}^{n} X_{i-1}^2 - \mathbf{E}\Big(\sum_{i=1}^{n} X_{i-1}^2 \Big)\Big|  \geq x \bigg)  \leq   \exp\bigg\{  - \frac{x^2}{2 n \frac{2  |\theta| +1}{1-\theta^2}  M^2 }      \bigg \}.
 \]
Recall that $ \sigma^{-2}\mathbf{E} S_n^2  \sim \frac{n \sigma^2 }{1- \theta^2}$ as $n\rightarrow \infty$.
The last inequality implies that  for all $x>0,$
\[
 \mathbf{P}\bigg( \Big| \frac{ \langle S\rangle_n}{\mathbf{E}S_n^2} - 1   \Big|  \geq x \bigg)  \leq   \exp\bigg\{  - C_{M, |\theta|, \sigma} n x^2       \bigg \} .
 \]
 By (\ref{eqD1}),    we have
\[
 \frac{1}{\sigma}(\hat{\theta}_n-\theta) \sqrt{\sum_{i=1}^n X_{i-1}^2 }=   \frac{S_n}{ \sqrt{\langle S\rangle _n} }.
\]
Clearly, $(\eta_i/\sqrt{\mathbf{E}S_n^2},\mathcal{F}_{i})_{i\geq 1}$ satisfies the conditions (A1)  and (A2) with $\epsilon_n = O(\frac {1} { \sqrt n})$ and $\delta_n = O(\frac1{\sqrt{n}})$. Thus the desired inequality follows from  Theorem \ref{ththeta2}.

\subsection{Proof of Corollary \ref{c0kls}}
By  Theorem \ref{ththeta2},  we have  for all $0\leq x=o (n^{1/6}),$
\begin{equation} \label{tphisns4}
\frac{\mathbf{P}\Big( \frac{1}{\sigma}(\hat{\theta}_n -\theta  )\sqrt{ \Sigma_{k=1}^n X_{k-1}^2} > x \Big)}{1-\Phi \left( x\right)}=1+o(1)\ \ \  \textrm{and}\ \ \  \frac{\mathbf{P}\Big( \frac{1}{\sigma}(\hat{\theta}_n -\theta  )\sqrt{ \Sigma_{k=1}^n X_{k-1}^2}  <-x \Big)}{ \Phi \left(- x\right)}=1+o(1).
\end{equation}
 Clearly, by (\ref{gdhddh}), the upper $(\kappa_n/2)$th quantile of a standard normal distribution
\[
\Phi^{-1}( 1-\kappa_n/2)=-\Phi^{-1}(   \kappa_n/2) = O(\sqrt{| \ln \kappa_n |} )
\]
 is of order $o\big( n^{1/6}  \big).$
Applying the last equality to (\ref{tphisns4}), we have
\begin{equation}
\mathbf{P}\Big( \frac{1}{\sigma}(\hat{\theta}_n -\theta  )\sqrt{ \Sigma_{k=1}^n X_{k-1}^2}   > \Phi^{-1}( 1-\kappa_n/2)\Big) \sim \kappa_n/2
\end{equation}
and
\begin{equation}
 \mathbf{P}\Big( \frac{1}{\sigma}(\hat{\theta}_n -\theta  )\sqrt{ \Sigma_{k=1}^n X_{k-1}^2} <-\Phi^{-1}( 1-\kappa_n/2) \Big) \sim \kappa_n/2,\ \ \ n\rightarrow \infty.
\end{equation}
Clearly, $\frac{1}{\sigma}(\hat{\theta}_n -\theta  )\sqrt{ \Sigma_{k=1}^n X_{k-1}^2} \leq \Phi^{-1}( 1-\kappa_n/2)$ implies that $\theta \geq A_n,   $ while $\frac{1}{\sigma}(\hat{\theta}_n -\theta  )\sqrt{ \Sigma_{k=1}^n X_{k-1}^2} \geq -\Phi^{-1}( 1-\kappa_n/2)$
means  $\theta \leq B_n.  $ Thus $[A_n, B_n]$  is a $1-\kappa_n$ confidence interval for $\theta$, for $n$ large enough.

\subsection{Proof of Corollary \ref{c0kdldss}}
By Theorem \ref{ththeta2}, we have for all $0\leq x=o (n^{1/2}),$
 \begin{equation}\label{ggsgsdf11}
\frac{\mathbf{P}\Big( \frac{1}{\sigma}(\hat{\theta}_n -\theta  )\sqrt{ \Sigma_{k=1}^n X_{k-1}^2} > x \Big)}{1-\Phi \left( x\right)}=\exp\bigg\{ \theta_1 C  \frac{ \ln n +x^{3}}{ \sqrt{n} } \bigg \}
\end{equation}
and
\begin{equation}\label{ggsgsdf12}
\frac{\mathbf{P}\Big( \frac{1}{\sigma}(\hat{\theta}_n -\theta  )\sqrt{ \Sigma_{k=1}^n X_{k-1}^2}  < - x \Big)}{1-\Phi \left( x\right)}=\exp\bigg\{ \theta_2 C  \frac{ \ln n +x^{3}}{ \sqrt{n} } \bigg \},
\end{equation}
 where $\theta_1, \theta_2 \in [-1, 1]$. Notice that
\[
1-\Phi \left( x_n\right) \sim \frac{1}{x_n\sqrt{2\pi}}e^{-x_n^2/2}= \exp\bigg\{-\frac{x_n^2}{2}\Big(1+\frac{2}{x_n^2}\ln (x_n\sqrt{2\pi})   \Big) \bigg\} ,\ x_n \rightarrow \infty.
\]
Thus, when $k_n \rightarrow 0$,
the upper $(\kappa_n/2)$th quantile of the distribution
\[
1-\Big(1-\Phi \left( x\right)\Big)\exp\bigg\{ \theta_1 C  \frac{ \ln n +x^{3}}{ \sqrt{n} } \bigg\}
\]
has an order of $ \sqrt{ 2 |\ln (\kappa_n/2)|}$, which by (\ref{3.3sfs}) is of order $o\big(\sqrt{n} \big)$  as $n\rightarrow \infty.$
Then   by an argument similar to the proof of Corollary \ref{c0kls}, we obtain
the desired result.

\section*{Acknowledgements}
Fan is deeply indebted to Prof.\ Ion Grama and Prof.\ Quansheng Liu for introducing his to the study of martingale limit theory.

\end{document}